%% file: InfinitesimalRigidity-V4.tex
\documentclass[11pt]{article}
\usepackage{amssymb,amsfonts,amsmath,latexsym,verbatim,amscd,amsthm,graphicx,epsfig}
\usepackage{array}
\usepackage{hyperref}
\usepackage{color}

\newtheorem{Prop}{Proposition}

\newtheorem{Def}{Definition}

\newtheorem{Thm}{Theorem}
\newtheorem{Cor}{Corollary}

\newtheorem{Rem}{Remark}

\newcommand{\R}{\mathbb{R}}
\newcommand{\Z}{\mathbb{Z}}

\newcommand{\K}{\kappa}
\newcommand{\n}{\nabla}
\newcommand{\na}{\nabla ^{*}}
\newcommand{\tr}{\mathrm{tr\,}}

\newcommand{\<}{\langle}
\renewcommand{\>}{\rangle}
\renewcommand{\d}{\partial}
\newcommand{\RO}{\mathring{R}}

\newcommand{\cs}{\mathrm{cs}}
\newcommand{\sn}{\mathrm{sn}}
\newcommand{\CM}{{\mathcal{CM}}}
\newcommand{\SG}{{\mathcal{SG}}}

\newcommand{\reg}{{\mathrm{reg}}}
\newcommand{\pr}{{\mathrm{pr}}}
\newcommand{\dev}{{\mathrm{dev}}}

\newcommand{\HH}{\mathbb H}
\newcommand{\RR}{\mathbb R}
\newcommand{\Sph}{\mathbb S}
\newcommand{\ZZ}{\mathbb Z}
\newcommand{\NN}{\mathbb N}

\newcommand{\CC}{\mathbb C}

\newcommand{\calA}{\mathcal A}
\newcommand{\calC}{\mathcal C}
\newcommand{\calD}{\mathcal D}
\newcommand{\calE}{\mathcal E}

\newcommand{\calG}{\mathcal G}

\newcommand{\calI}{\mathcal I}
\newcommand{\calL}{\mathcal L}
\newcommand{\calM}{\mathcal M}
\newcommand{\calN}{\mathcal N}
\newcommand{\calO}{\mathcal O}
\newcommand{\calP}{\mathcal P}
\newcommand{\calS}{\mathcal S}

\newcommand{\calU}{\mathcal U}
\newcommand{\calV}{\mathcal V}

\newcommand{\del}{\partial}

\newcommand{\Ric}{\mathrm{Ric}}

\newcommand{\e}{\epsilon}

\newcommand{\phg}{{\mathrm{phg}}}

\newcommand{\Fr}{{\mathrm{Fr}}}
\newcommand{\DN}{{\mathrm{DN}}}

\newcommand{\ie}{\mathrm{ie}}
\newcommand{\calR}{\mathcal R}
\newcommand{\edge}{\mathrm{e}}

\title{Infinitesimal rigidity of cone-manifolds \\
and the Stoker problem for hyperbolic and {E}uclidean polyhedra}
\author{Rafe Mazzeo \thanks{Partially supported by the NSF grant DMS-0505709 and 0805529} \\ Stanford University \and 
Gr\'egoire Montcouquiol \thanks{Partially supported by the ANR program GeomEinstein 06-BLAN-0154} \\ Universit\'e Paris-Sud}
\date{}

\begin{document}

\maketitle

\begin{abstract}
The deformation theory of hyperbolic and Euclidean cone-mani\-folds with all cone angles less than $2\pi$ plays
an important role in many problems in low dimensional  topology and in the geometrization of $3$-manifolds.
Furthermore, various old conjectures dating back to Stoker about the moduli space of convex hyperbolic and Euclidean
polyhedra can be reduced to the study of deformations of cone-manifolds by doubling a polyhedron across
its faces. This deformation theory has been understood by
Hodgson and Kerckhoff \cite{HK} when the singular set has no vertices, and by Wei\ss{} \cite{Weiss2}
 when the cone angles are less than $\pi$. We prove here an infinitesimal rigidity result valid for cone angles less than $2\pi$, 
stating that infinitesimal deformations which leave the dihedral angles fixed are trivial in the hyperbolic case, and reduce to
some simple deformations in the Euclidean case.  The method is to treat this as a problem concerning the
deformation theory of singular Einstein metrics, and to apply analytic methods about elliptic operators on
stratified spaces.  This work is an important ingredient in the local deformation theory of cone-manifolds by
the second author \cite{Montcouq2}, see also the concurrent work by Wei\ss{} \cite{Weiss4}. 
\end{abstract}

\section{Introduction}
A conjecture made by J.J.\ Stoker in 1968 \cite{Stoker} concerns the rigidity of convex polyhedra in three-dimensional 
constant curvature spaces. More specifically, he asked whether such a polyhedron is determined by its dihedral angles, 
i.e.\ the angles between its faces. In the Euclidean case, the existence of homotheties and translations of faces obviously 
contradicts rigidity, and Stoker asked the more precise question of whether the \emph{internal} angles of each face are 
determined by the set of all dihedral angles of the polyhedron. On the other hand, in the spherical and hyperbolic settings, 
there are no obvious obstacles to the validity of such a rigidity phenomenon. 

This problem has motivated many papers (see for instance \cite{Karcher}, \cite{Milka}, \cite{Pog}), but until recently results 
had only been obtained for specific classes of polyhedra. Surprisingly, Schlenker \cite{SchlenkerPoly} found a counter-example 
to both the  infinitesimal and global versions of this conjecture in the spherical case. In this paper we settle the infinitesimal version 
of the two remaining cases; this states that there is no nontrivial deformation of a convex hyperbolic polyhedron $\calP$ for which the 
infinitesimal variation of all dihedral angles vanishes. In Euclidean space, there exist such first-order deformations but they preserve 
the internal angles of the faces. 

We work in the somewhat more general context of (hyperbolic or Euclidean) cone-manifolds. This is a class of constant
curvature stratified spaces emphasized, if not actually introduced, by Thurston \cite{ThurstonGeom} in his investigation 
of deformations of cusped hyperbolic $3$-manifolds. These are Riemannian generalizations of orbifolds in that the cone angles
at each edge are arbitrary positive numbers (possibly larger than $2\pi$), and in particular are not necessarily of 
the form $2\pi/k$, $k \in \NN$, as they are for orbifolds. There is a completely analogous problem, which we also
call the Stoker conjecture, concerning the rigidity of hyperbolic or Euclidean cone-manifolds, and we resolve the 
infinitesimal version of this too assuming that the cone angles at all edges are less than $2\pi$. This condition 
on the cone angles is the analogue of convexity in this setting. In fact, any polyhedron
$\calP$ can be simultaneously doubled across all its faces, giving rise to a cone-manifold with the same curvature. 
Convexity of the polyhedron is equivalent to this angle condition on its double. 

\begin{Thm}[The Infinitesimal Stoker Conjecture for Cone-Manifolds]\label{thm:mainthm}
Let $M$ be a closed three-dimensional cone-manifold with all cone angles smaller than $2\pi$. If $M$ is hyperbolic, then 
$M$ is infinitesimally rigid relative to its cone angles, i.e.\ every angle-preserving infinitesimal deformation is trivial.
If $M$ is Euclidean, then every angle-preserving deformation also preserves the spherical links of the codimension $3$ 
singular points of $M$.

In particular, convex hyperbolic polyhedra are infinitesimally rigid relatively to their dihedral angles, while every 
dihedral angle preserving infinitesimal deformation of a convex Euclidean polyhedron also preserves the internal angles 
of the faces.
\end{Thm}

The cone-manifold rigidity problem has been investigated before, see in particular Hodgson and Kerckhoff's seminal 
article \cite{HK} for the case were the singular locus is a link, and Wei\ss{}'s paper \cite{Weiss2} which treats the case where all 
dihedral angles are smaller than $\pi$. 

Our approach to these problems is based on global analysis, and in particular on regarding cone-manifolds as Einstein 
manifolds with very special types of metric singularities. The infinitesimal rigidity statement reduces eventually to a Bochner 
argument, but the first main step is to put an infinitesimal deformation which preserves dihedral angles into a good
gauge so that it has a tractable form near the singular locus. We employ a now standard formalism for Einstein 
deformation theory using the Bianchi gauge to do this. The second main step is to study the asymptotics of this gauged
infinitesimal deformation $h$ in order to justify the integrations by parts in the Bochner argument. This is equivalent to 
showing that $h$ lies in a suitable (functional analytic) domain for a semibounded self-adjoint extension of the linearized 
gauged Einstein operator, which involves a regularity theorem showing that $h$ has an asymptotic expansion and an 
examination of the leading terms in this expansion. The existence of such an expansion is now standard for manifolds 
with isolated conic or simple edge singularities. However, three-dimensional cone-manifolds may have slightly more 
complicated (``depth $2$'') singular structure, so this requires some new analytic work. 

The plan of this paper is as follows. We begin with a review of cone-manifolds, with an emphasis on the two-and 
three-dimensional settings, in the context of the more general notion of iterated cone-edge spaces. We then discuss 
the deformation theory of cone-manifolds via the less obstructed problem of deforming the germ of the 
cone-manifold structure in a neighbourhood of the singular locus. This leads to a precise statement of the 
infinitesimal deformation problem. After that we review the analytic tools needed to study the deformation problem 
on the entire cone-manifold: first, the Einstein equation and Bianchi gauge, and their linearizations, then the analytic
theory of conic and iterated edge operators. With these tools we carry out the remainder of the proof.

Since the time we proved this result, but before this paper was written, the second author has studied the local deformation
theory and obtained a corresponding local (rather than infinitesimal) rigidity result for $3$-dimensional hyperbolic 
cone-manifolds and polyhedra \cite{Montcouq2}; this uses the infinitesimal rigidity proved here.  At the same time, Wei\ss{} 
has independently proved a similar local rigidity theorem using somewhat different methods \cite{Weiss4}. 

The authors wish to thank Steve Kerckhoff for his interest; the first author is also grateful to Hartmut Wei\ss{} for 
helpful discussions. The authors are also indebted to the three referees, who read the manuscript carefully
and gave many valuable suggestions to improve the exposition. 

\section{Geometry of cone-manifolds}\label{sec:geomconemfd}

Cone-manifolds can be defined synthetically as $(X,G)$-spaces in Thurston's sense, but the point of view we first adopt here is to treat them as the constant curvature elements of the class of Riemannian iterated edge spaces, or conifolds. We begin with a brief review of 
this latter class of singular spaces, then recall the synthetic description of cone-manifolds, relating the 
two descriptions along the way.

\subsection*{Iterated edge spaces} 
Let $(N,h)$ be a compact stratified Riemannian space. This means simply that $N$ is a smooth stratified space 
in the usual sense, cf.\ \cite{Pflaum}, that each stratum $S$ carries a Riemannian metric $h_S$, and that 
these various metrics satisfy the obvious compatibility relationships. (We do not belabour this definition 
because all examples considered here will be quite simple.) The (complete) \emph{cone over $N$}, $C(N)$ is the space 
$\left([0,\infty)_r \times N\right)/\sim$, where $(0,p) \sim (0,p')$ for all $p,p' \in N$, endowed with the 
metric $dr^2 + r^2 h$; the truncated cone $C_{a,b}(N)$ is the subset where $a \leq r \leq b$. Any singular 
stratum $S \subset N$ induces a singular stratum $C(S)$ in $C(N)$, with $\dim C(S) = \dim S + 1$. 

The class of iterated edge spaces, which we also call conifolds, are those which can be obtained locally by iterated coning 
and formation of products, starting from smooth compact manifolds. More formally, for each $k \geq 0$ the class 
$\calI_k$ of compact conifolds of depth $k$ is defined by induction as follow:

\begin{Def} A conifold of depth $0$ is a compact smooth manifold.  A stratified pseudomanifold $X$ lies in $\calI_k$ 
if it is a compact stratified space and any point $p \in X$, contained in an open stratum $S$ of dimension $\ell$, has a 
neighbourhood $\calU$ such that 
$\calV = \calU \cap S$ is diffeomorphic to an open ball in $\R^\ell$ and such that $\calU$ is diffeomorphic to a product 
$\calV \times C_{0,1}(N)$ where $N \in \calI_{j}$ for some $j < k$. The dimension of $X$ is defined by induction as 
$n = \ell + \dim N +1$; we ask furthermore that this quantity is independent of the point $p \in X$.
An incomplete iterated edge metric $g$ on $X$ is one which respects the above diffeomorphism, i.e.\ is locally quasi-isometric 
to one of the form $g \sim dr^2 + r^2 h + \kappa$, where $h$ is an incomplete iterated edge metric on $N$ and $\kappa$ is a 
metric on $S$. The entire class of conifolds $\calI$ is the union over $k \in {\mathbb N}$ of these subclasses $\calI_k$.
\end{Def}

Thus $X \in \calI_k$ if it can be formed by a $k$-fold iterated coning or edging procedure. A point for which $N=\emptyset$ 
is called \emph{regular}; it lies in the top-dimensional stratum of $X$. If $S$ is a singular stratum with $\dim S > 0$, then 
we say that it is an edge in $X$; some neighbourhood of $S$ in $X$ is diffeomorphic to a bundle of cones over $S$ with fibre 
$C(N)$. For various investigations of analysis on these spaces it is usually necessary to assume more about
the structure of the metric near the singular strata, but we do not elaborate on this here since our goals are more limited.

A careful discussion of the differential topology of this class of spaces, including a comparison with more classical definitions of 
stratified pseudomanifolds, and the construction of a resolution of any iterated edge space as a manifold with corners with ``iterated 
fibration structure'' is contained in the recent paper \cite{ALMP}. 

A (constant curvature) \emph{cone-manifold} is a conifold $(M,g)$ such that the induced 
metric $g_S$ on any stratum of dimension $>1$ has constant sectional curvature $\K$ (in particular, this is true on its top-dimensional 
stratum, which is an open and dense subset).  The name `cone-manifold' is somewhat misleading, but 
it is the standard and accepted terminology so we do not discard it. A cone-manifold is called hyperbolic, Euclidean or 
spherical depending on whether $\K = -1$, 
$0$ or $+1$ (any other $\K$ can be reduced to one of these three cases by rescaling the metric). We always denote by 
$\Sigma$ the entire \emph{singular locus} of $M$, i.e.\ the union of all strata with dimension less than $\dim M$.

The two- and three-dimensional settings are the ones of interest in this paper. Thus, a cone-surface $N$ is a 
two-dimensional space with isolated conic singularities (hence an element of $\calI_1$) and with constant Gauss 
curvature $\K$ on its smooth part. Near each singular point $p$ it is a (constant curvature) cone over a circle,
with metric 
\begin{equation}
dr^2 + \sn_{\K}^2 r \, d\theta^2, \qquad \theta \in \RR / \alpha \ZZ \equiv \Sph^1_\alpha.
\label{eq:metricconesurface}
\end{equation}
The number $\alpha$ is called the \emph{cone angle} at $p$. Here and throughout the paper we use the convention that $\sn_{\K}$ 
and $\cs_{\K}$ are the unique solutions to the initial value problem $f'' + \K f = 0$ satisfying
\begin{align*}
\sn_{\K} (0) &= 0 & \cs_{\K} (0) &= 1\\
\sn_{\K}'(0) &= 1 & \cs_{\K}' (0) &= 0.
\end{align*}

Accordingly, a three-dimensional cone-manifold is an element $M \in \calI_2$ with $\dim M = 3$ and with constant sectional curvature
$\K$ on its top-dimensional stratum. Its singular locus $\Sigma$ is a (combinatorial) graph which decomposes as a disjoint union 
$\Sigma_0 \cup \Sigma_1$; here $\Sigma_0 = \calV$ is the vertex set and consists of a finite number of points,
and $\Sigma_1 = \calE$ is the edge set and consists of a finite union of smooth arcs each of which is either closed 
or else has endpoints lying in $\Sigma_0$. It is also useful to think of $\Sigma$ as a geodesic network. 
Near any point $p \in \Sigma_1$, the metric $g$ can be written in cylindrical coordinates as
\begin{equation}
d\rho^2 + \sn_{\K}^2 \rho \, d\theta^2 + \cs_{\K}^2 \rho \, dy^2, \qquad \theta \in \Sph^1_\alpha,\ y \in (-a,a) \subset \RR.
\label{eq:metric1conemfd}
\end{equation}
This is a `constant curvature cylinder' over the constant curvature cone $C_{0,1}(\Sph^1_\alpha)$. The number $\alpha$ is still called the \emph{cone angle} or \emph{dihedral angle} of the singular edge, and does not depend of the chosen point $p$ on the edge. Near a singular vertex $p \in \Sigma_0$,
$M$ is a constant curvature cone over a spherical cone-surface $(N,h)$, and here $g$ has the form 
\begin{equation}
dr^2 + \sn_{\K}^2 r \, h.
\label{eq:metric2conemfd}
\end{equation}
The metric $h$ in turn has the form (\ref{eq:metricconesurface}) with $\K = +1$ near each one of its singular points,
so in a conic neighbourhood of a cone point of $N$, the metric $g$ on $M$ has the form
\begin{equation}
dr^2 + \sn_{\K}^2 r \left(ds^2 + \sin^2 s\, d\theta^2\right), \qquad \theta \in \Sph^1_\alpha,\ s \in (0,\epsilon),\ r \in (0,a).
\label{eq:metric3conemfd}
\end{equation}
In particular, each cone point of $N$ corresponds to an edge of $M$.
 
Since we are only concerned with these low-dimensional cases in this paper, the terms cone-surface 
and cone-manifold will always refer to the two- and three-dimensional cases, respectively. 

\subsection*{Synthetic formulation}
We now review the more traditional definitions of cone-surfaces and cone-manifolds `modeled on a geometry', 
following \cite{ThurstonShapes}, but cf.\ the monograph by Hodgson and Kerckhoff \cite{HK2} for an alternate 
and excellent reference on hyperbolic cone-manifolds.  

A {\it geometry} refers to a pair $(X,G)$ consisting of a complete Riemannian $n$-manifold $X$ and a subgroup 
$G \subset \mbox{Isom}\,(X)$. There are no restrictions on $G$, and in particular, it need not be the
full isometry group, nor must it act transitively on $X$. An $(X,G)$-manifold $M$ is a manifold admitting an 
$(X,G)$-atlas, i.e.\ a locally finite open covering $\{\calU_i\}$ of $M$ and maps $\phi_i: \calU_i \to X$, 
such that the transition maps $\phi_i \phi_j^{-1}$ all lie in $G$. Since the elements of $G$ are isometries, 
$M$ possesses a natural Riemannian metric for which each $\phi_i$ is an isometry. 

Given a $n$-dimensional geometry $(X,G)$ ($n>1$) and a point $q$ on $X$, denote by $G_q = \{ g \in G \ |\ g q = q\}$ 
the stabilizer of $q$ in $G$ and $X_q = \{ v \in T_qX \ |\ ||v||=1\}$, the unit sphere in $T_q X$. The group $G_q$ acts 
by isometries on $X_q$, so the pair $(X_q,G_q)$ defines a new geometry, of dimension $n-1$. If $\Omega$ is a subset of $X_q$, 
then for $a$ smaller than the injectivity radius of $X$ at $q$ we define the $(X,G)$-cone of radius $a$, 
\[
C_a(\Omega)=\{\exp_q(tv), 0\leq t < a, v \in \Omega\} \subset X.
\]
Any transition map $\psi \in G_q$ which glues together neighbourhoods $\Omega_1$ and $\Omega_2$ extends in an obvious way 
to a transition map gluing the $(X,G)$-cones $C_a(\Omega_1)$ and $C_a(\Omega_2)$. 
\begin{figure}[ht]
\scalebox{0.5}{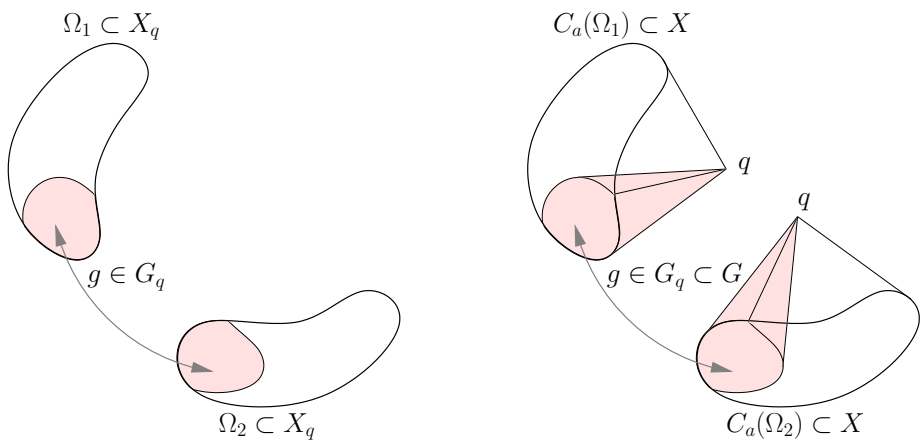}
\caption{Extending a gluing from $X_q$ to $X$}
\end{figure}
This process allows us to define the $(X,G)$-cone associated to any $(X_q, G_q)$-manifold, or inductively, 
any $(X_q,G_q)$-cone-manifold, and thus leads to the following definition:

\begin{Def}
Let $(X,G)$ be a geometry as defined above.
\begin{itemize}
\item If $X$ is of dimension $1$, an $(X,G)$-cone-manifold is an $(X,G)$-manifold; 
\item If $\dim X > 1$, an $(X,G)$-cone-manifold is a complete metric space in which each point $p$ has a neighbourhood 
isometric to an $(X,G)$-cone over a closed, connected, orientable $(X_q,G_q)$-cone-manifold $N_p$ for some $q\in X$; 
the lower dimensional cone-manifold $N_p$ is called the {\em link} of the point $p$.  (The reader should be aware
that this does not coincide with the more standard definition of the link of a cone bundle.)
\end{itemize}
\end{Def}
 
Even if $(X,G)$ is homogeneous, i.e.\ $G$ acts transitively on $X$, this may not be the case for the induced geometry $(X_q,G_q)$, hence it is important not to include the statement that $X$ be a homogeneous space into the definition of a geometry. 

\begin{figure}[ht]
\centering
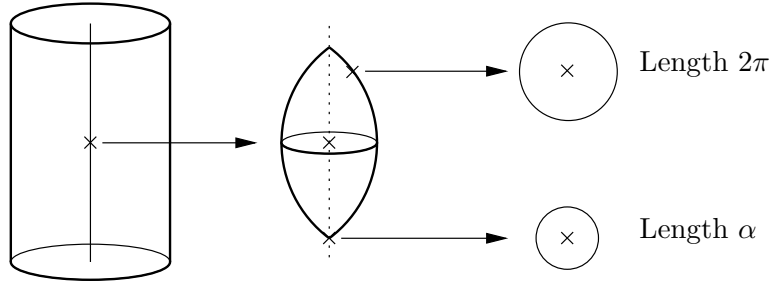
\caption{The inductive steps in the definition of a cone-manifold}
\label{ind}
\end{figure}

An $(X,G)$-cone-manifold is actually an iterated edge space with an additional structure. 
A point in an $(X,G)$-cone-manifold $M$ is regular (i.e.~in the top-dimensional stratum) if its link is the standard unit sphere; otherwise it is 
in the singular locus. Equivalently, a point is regular if it admits a neighbourhood isometric to an open set in $X$. 
The stratification of the singular locus of $M$ is defined by induction as follows. If $p \in M$ is not contained in 
any cone neighbourhood other than the one centered at $p$, then the codimension of $p$ is $n$. Otherwise, $p$ belongs 
to a cone neighbourhood centered at some other point $p_0$. If $N_0$ is the $(X_q,G_q)$-link of $p_0$, then the geodesic ray 
from $p_0$ through $p$ determines a point $p'$ in $N_0$. The codimension of $p$ is then defined inductively as the codimension of $p'$ 
in $N_0$ (in this case the link $N_p$ of $p$ is a suspension over the link $N_{p'}$ of $p'$ in $N_0$). Note that cone-manifolds of dimension $1$
are by definition regular, so the codimension of a singular point is always at least two. The set of points of codimension $k$
forms the $(n-k)$-dimensional stratum $\Sigma_k$; the strata are totally geodesic in $M$, and are 
locally isometric to some $(n-k)$-dimensional totally geodesic submanifold of $X$.

\medskip
The cases of interest in this paper are the following:
\begin{itemize}
\item spherical cone-surfaces, modeled on $(\Sph^2,\mbox{SO}(3))$;
\item three-dimensional hyperbolic cone-manifolds, modeled on $(\mathbb{H}^3,\mbox{SO}_0(1,3))$;
\item three-dimensional Euclidean cone-manifolds, modeled on $(\RR^3, \RR^3 \rtimes \mbox{SO}(3))$.
\item three-dimensional spherical cone-manifolds, modeled on $(\Sph^3,\mbox{SO}(4))$.
\end{itemize}
Note that for each of these three-dimensional geometries, 
the induced link geometry is always spherical; Euclidean or hyperbolic cone-surfaces do not arise here. If 
the link is one-dimensional, the only invariant is its length, i.e.\ the cone angle; for two-dimensional links,
the `solid angle analogue' of cone angle {\it is} the spherical cone-surface structure on the link. 
We elaborate on this below.

In this paper, we restrict attention to Euclidean and hyperbolic cone-manifolds whose cone angles are less 
than $2\pi$. We included the other cases since much of the analysis below carries over directly to these settings, 
although the ultimate infinitesimal rigidity theorem does not.

\subsection*{Spherical cone-surfaces}
As just indicated, to study cone-manifolds with all cone angles less than $2\pi$, we must first
study spherical cone-surfaces with the same angle restriction. This is an interesting story in
its own right, and is described in detail in \cite{MW}. Here are the relevant features of that theory.

The first important fact is that if $N$ is any (compact) cone-surface, then its regular part is 
conformally equivalent to a compact Riemann surface $\overline{N}$ with $\ell$ points removed. (Equivalently, 
although the cone metric is singular at each of the cone points, the conformal structure extends smoothly 
across these points.) This is classical, but see \cite{MT} for a simple proof. 

Next, let $N$ be a spherical cone-surface with cone points $\{q_i\}$ and corresponding cone angles $\alpha_i 
\in (0,2\pi)$, $i = 1, \ldots, \ell$. We claim that if $N$ is orientable, then $\overline{N} = \Sph^2$. 
To see this, recall that the Gauss-Bonnet formula extends to cone-surfaces as
\[
\int_N \K\,dA + \sum_{i=1}^\ell (2\pi - \alpha_i) = 2\pi \chi(\overline{N}).
\]
Since $\K = 1$ and each $\alpha_i < 2\pi$, the left hand side of this equality is strictly positive, 
hence $\chi(\overline{N}) > 0$, which proves the claim.  

For any collection of $\ell \geq 3 $ distinct points $\{q_1, \ldots, q_\ell\} \subset \Sph^2$ and cone 
angles $\alpha_i \in (0,2\pi)$,  there exists a unique spherical cone metric on $\Sph^2$ in the given 
conformal class with cone points at the $q_i$ with the specified cone angles. The existence is due to Troyanov \cite{Tr},
and uniqueness was proved by Luo and Tian \cite{LT}. We denote by $\calM_{\ell}(\Sph^2)$ the moduli space 
of all spherical cone-surface structures (with curvature $+1$) with $\ell$ marked points on $\Sph^2$ and with 
cone angles less than $2\pi$. The result we have just quoted is captured in the identification
\[
\calM_{\ell}(\Sph^2) = \{(q_1, \ldots, q_\ell) \in (\Sph^2)^\ell\ :\  q_i \neq q_j\ \forall \, i \neq j\}/
\mbox{M\"ob} \times (0,2\pi)^\ell, \qquad \ell \geq 3,
\]
where the M\"obius group acts diagonally on the product. Notice that we are not dividing by the action of the 
symmetric group, i.e.\ we regard the cone points as labelled. 

We discuss the cases with few cone points separately. 
It is classical that a spherical cone-surface (with cone angles smaller than $2\pi$) must have at least two cone points, and that $\calM_{2}(\Sph^2) \cong (0,2\pi)$. More explicitly, if $N \in \calM_{2}(\Sph^2)$, then 
the cone metric has an $\Sph^1$-symmetry, so that $\alpha_1 = \alpha_2$, and $N$ has the shape of a football (or 
rugby ball), see figure \ref{ind};  the metric can be expressed globally as \eqref{eq:metricconesurface} with $r\in [0,\pi]$. The three-dimensional cone with constant 
curvature $\K$ over such an $N$ is isometric to the cylinder with constant curvature $\K$ over the two-dimensional 
cone with constant curvature $\K$ (and with the same cone angle), see eq.~\eqref{eq:metric1conemfd}. 
Hence this case never occurs as the link of a 
(`nonremovable') vertex in a cone-manifold. Thus the first really interesting case in terms of three-dimensional 
geometry is when $N$ has three conical points. In this case, we can assume that all points lie on an equator of 
$\Sph^2$ (in the above identification), and by uniqueness, the metric is preserved by reflection across this equator, so that $N$ is 
the double of a spherical triangle with geodesic edges. In particular, $N$ is determined entirely by its cone angles. 

\subsection*{Three dimensional cone-manifolds}
By considering the spherical reduction and link geometry, it is quite easy to see that a neighbourhood of any 
point $p$ in a three-dimensional cone-manifold $M$ with constant curvature $\K$ is isometric to a neighbourhood of a point 
in a space obtained by gluing together some number of tetrahedra with constant curvature $\K$ and with
totally geodesic boundary faces. In such a gluing, the metric extends smoothly across 
the codimension $1$ boundary faces. At points where several edges are glued together, the metric is smooth 
if and only if the sum of the corresponding dihedral angles for each tetrahedron is equal to $2\pi$, otherwise 
the resulting edge is in the singular locus of the cone-manifold and a neighbourhood of this edge is isometric 
to a constant curvature cylinder (\ref{eq:metric1conemfd}). The sum of the dihedral angles is called the
\emph{dihedral angle} (or cone angle) of this edge. A point where several vertices are glued together is 
smooth if and only if the links of these tetrahedral vertices tile $\Sph^2$; otherwise, the link is a 
spherical cone-surface $N$ and must lie 
in the space $\calM_{\ell}(\Sph^2)$ for some $\ell \geq 3$ which is the valence of that vertex in the cone-manifold. 
The cone angle at each cone point of $N$ equals the dihedral angle of the edge coming into that point.

More globally, it follows from the definition that the singular locus $\Sigma$ of a closed three-dimensional 
cone-manifold $M$ is a finite graph, where each vertex is at least trivalent, but possibly with both
endpoints of an edge equal to the same vertex. We have described local geometry around each edge and
vertex, and away from $\Sigma$ the metric is locally isometric to the model geometry.

An important and simple case of the gluing construction is as follows. Let $\calP$ be a convex polyhedron in
$\Sph^3$, $\RR^3$ or $\HH^3$. Its double across all faces is a (spherical, Euclidean or hyperbolic) cone-manifold 
$M$. The edges and vertices of $\calP$ are in bijective correspondence with the edges and vertices of $M$. The
dihedral angle along an edge of $M$ is twice that of the corresponding edge in $\calP$, so the convexity of the
original polyhedron corresponds to the fact that all cone angles of $M$ are less than $2\pi$. The spherical cone 
surface at any vertex $p$ in $M$ is the double of a convex spherical polygon in $\Sph^2$.  Consequently,
rigidity results for $3$-cone-manifolds imply corresponding results for polyhedra. 

\section{Geometric deformation theory}
We now examine the more geometric aspects of the deformation theory of cone-manifolds. The main idea is to localize
near $\Sigma$ and study the deformations of a `tubular neighbourhood' $\calU$ around the singular locus. 
These local cone-manifold structures, which we call \emph{singular germs}, have unobstructed deformation
theory in a sense to be made precise later, hence are quite easy to parameterize. Any cone-manifold determines
a singular germ along its singular locus, but the converse problem of extending a singular germ to
a global compact cone-manifold is much more subtle. We shall return to this question in a later paper. 
Here we use these singular germs as a convenient setting to study the infinitesimal deformations
of cone-manifold structures. 

\subsection*{Singular germs}
Fix a compact, topological graph $\Sigma$ and define $\SG(\Sigma,\K)$ to be the set of all cone metrics defined in a neighbourhood 
of $\Sigma$ with curvature $\K$  and cone angles smaller than $2\pi$, modulo the equivalence relation that two such 
structures are identified if they are isometric (in possibly smaller neighbourhoods). We call any such equivalence class 
$\calS$ the \emph{singular germ} of a cone-manifold structure. Any singular germ $\calS$ is represented by an infinite 
dimensional set of mutually isometric cone metrics. We stress that we do not emphasize explicit uniformizations of 
constant curvature metrics but focus instead on classes of mutually isometric metric tensors. 

We now describe these singular germs synthetically. The singular locus is described entirely by its simplicial structure,
e.g.\ its vertex set $\calV$ and edge set $\calE$, and the length of each of its edges. (Note that unlike a one-dimensional 
simplicial complex in the usual sense, $\Sigma$ may have closed components which contain no vertices; however,
we still call $\Sigma$ a graph or a simplicial complex.) To describe a cone metric in a neighbourhood of its singular locus, we must extend it from the interior of each of the edges and also near the vertices, and these extensions 
must satisfy certain local and global compatibility conditions. 

The extension near (the interior of) any edge $e$ is determined uniquely by the curvature $\K$ and the cone angle $\alpha =
\alpha(e)$ associated to that edge,  as in (\ref{eq:metric1conemfd}). To each vertex $p$, we specify a spherical 
cone-surface $N_p \in \calM_{\ell}(\Sph^2)$, where $\ell$ is the valence of $p$. One additional parameter along each edge 
is needed to completely describe the singular germ. Any edge $e$ joins two  vertices $p_1, p_2 \in \calV$ (where the
case $p_1 = p_2$ is allowed), and determines cone points $q_1 \in N_{p_1}$ and $q_2 \in N_{p_2}$. These must both have 
cone angle $\alpha(e)$, or equivalently, the same `link circle' $\Sph^1_{\alpha(e)}$. Parallel transport along $e$ gives an 
orientation-reversing isometry between these two circles; after fixing one such isometry for each singular edge, any other such identification is identified with a rotation of the circle, i.e.\ a number $\tau(e) \in \Sph^1_{\alpha(e)}$. This is called 
the {\it twist parameter} associated to $e$.

\begin{figure}[ht]
\centering
\scalebox{0.8}{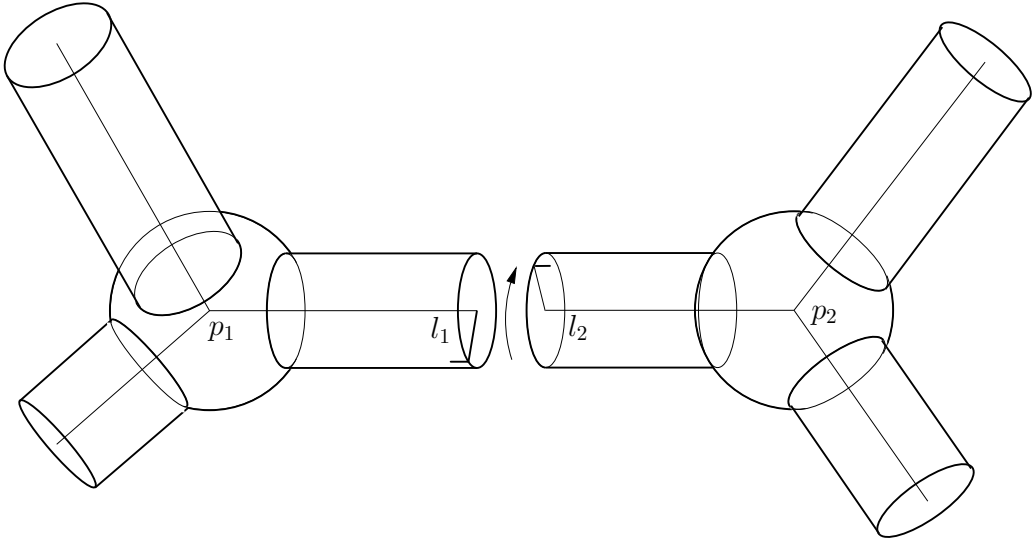}
\caption{The twist parameter}
\label{figtwist}
\end{figure}

Thus each singular germ corresponds to the following data: the curvature $\K$; the graph $\Sigma$ with vertex
set $\calV$ and edge set $\calE$; a length $\lambda(e)$, cone angle $\alpha(e)$ and twist parameter $\tau(e)$ for 
each $e \in \calE$; a spherical cone-surface $N_p$ to each vertex $p \in \calV$. Let $n(p)$ denote the valence of 
each vertex $p$ and set
\[
\calN(\Sigma) = \prod_{p \in \calV} \calM_{n(p)}(\Sph^2).
\]
Then a singular germ on a fixed graph $\Sigma$ with curvature $\K$ is encoded by some element of the space
\[
\SG^*(\Sigma,\K) :=  (\RR^+)^{|\calE|}_\lambda \times (0,2\pi)^{|\calE|}_{\alpha} \times [0,2\pi)^{|\calE|}_\tau \times \calN(\Sigma).
\]
The subset of elements which are realizable as singular germs is the space $\SG(\Sigma,\K)$ defined earlier; this has the 
explicit description as consisting of the subset of data for which the cone angles at the points $q_1\in N_{p_1}$ and $q_2\in N_{p_2}$ corresponding to the 
two ends $p_1$ and $p_2$ of any edge $e$ are the same, and equal $\alpha(e)$. We frequently omit the $\Sigma$ and $\K$ from this notation.
There is a natural map
\[
G: \CM \longrightarrow \SG
\]
associating to any cone-manifold $M$ the singular germ along its singular locus $\Sigma$. 

It is almost tautological to construct a local cone metric from any point in $\SG$. Indeed, given $(\Sigma,\K,\{\lambda(e)\}, 
\{\alpha(e)\}, \{\tau(e)\}, \{N_p\})$, choose $a > 0$ sufficiently small (in particular, $a < \frac12 \lambda(e)$ for all 
$e \in \calE$ and also less than $\pi/\sqrt{\K}$ if $\K > 0$). Now take the cone of curvature $\K$ and radius $a$ over the 
cone-surface $N_p$ for each vertex $p$; attach to these cones singular tubes of constant curvature $\K$  (of sufficiently small 
radius $b$) and of specified cone angles. The result is depicted 
in figure \ref{figtwist}. The final step is to glue these pieces together using the combinatorics of $\Sigma$ and the twist parameter along each edge and to ensure 
that each edge-length is the prescribed one.

\subsection*{Deformations of singular germs}
\label{defsinggerm}
Fix a singular germ $\calS \in \SG$ and choose a representative cone metric $g$. There are several ways to deform 
$\mathcal{S}$ in $\SG$:

\begin{enumerate}
\item One can change the curvature, or the simplicial structure of the singular locus (see fig. \ref{figsplit}). Such deformations 
are quite interesting, and are considered in \cite{Montcouq2}, \cite{PW} and \cite{Weiss4}; however, in this paper we
fix $\Sigma$ and $\K$ once and for all. 
\begin{figure}[ht]
\centering
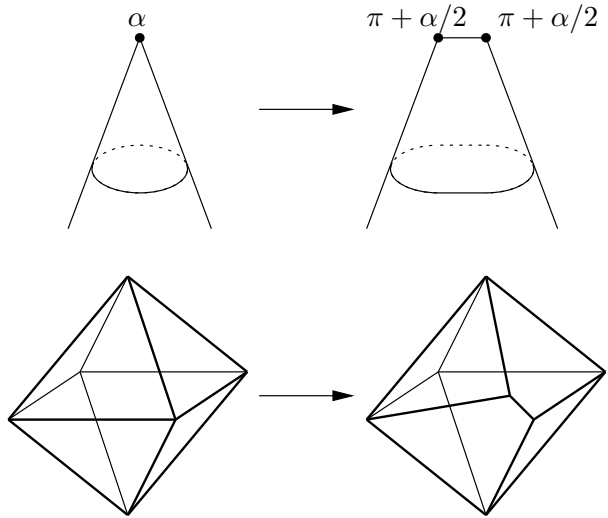
\caption{Deformation of the singular locus in dimension $2$ and $3$}
\label{figsplit}
\end{figure}

\item One can change the length $\lambda(e)$ of any edge of $\Sigma$. We present an explicit family of metrics which does this.
In cylindrical coordinates $(\rho,\theta,y)$ around any point in the interior of $e$, the metric has the form 
(\ref{eq:metric1conemfd}). Let $f(y)$ be a smooth nondecreasing function which vanishes for $y \leq \lambda/4$ and which 
equals $1$ for $y \geq 3\lambda/4$. Now define 
\begin{equation}
g_\e = d\rho^2 + \sn_{\K}^2 \rho\,  d\theta^2 + \cs_{\K}^2 \rho\, \left(1+ \e f'(y)\right)^2 dy^2
\label{eq:sf1}
\end{equation}
in the tube around this edge (and let $g_\e = g$ around all other edges and vertices). The length of $e$ with respect to
$g_\e$ is $\lambda + \e$.

\item One can change the twist parameter around an edge $e$ in much the same way. The family of metrics $g_\e$ is now given by 
\begin{equation}
g_\e = d\rho^2 + \sn_{\K}^2 \rho\, \left(d(\theta + \e f(y))\right)^2 + \cs_{\K}^2 \rho\, dy^2
\label{eq:sf2}
\end{equation}
near $e$ and is left the same elsewhere. The twist parameter for $e$ with respect to $g_\e$ equals $\tau + \e$. 

\item Finally, one can change the cone-surface structures at each of the vertices $p \in \Sigma$. We refer to \cite{MW} for a 
complete description of this moduli space theory of cone-surfaces, but as discussed earlier, the moduli parameters are the 
locations of the various cone points on each spherical link, and the cone angles at these points. Each moduli space 
$\calM_{\ell}(\Sph^2)$ is smooth, hence so is $\calN(\Sigma)$. Let $h_\e$ be a path of metrics representing a curve 
in this space; the corresponding path of metrics in cones around the vertex set is then given by 
\begin{equation}
g_\e = dr^2 + \sn_{\K}^2 r\, h_\e \label{eq:sf3}
\end{equation} 
If the deformation $h_\e$ leaves all cone angles unchanged, then the variation of $g_\e$ is localized in a neighbourhood 
of the vertex set and so $g_\e=g_0$ in a neighbourhood of the edge set. However, a change of cone angle in any $N_p$ 
ripples throughout the singular germ; in particular the cone angles of some edges are deformed, and the metric in cylindrical 
coordinates near those edges is given by
\begin{equation}
g_\e = d\rho^2 + \sn_{\K}^2 \rho\, (1+a(\e))^2d\theta^2 + \cs_{\K}^2 \rho\, dy^2
\label{eq:sf4}
\end{equation}
where $a(\e)$ is the modification of the cone angle.
\end{enumerate}

The expression \eqref{eq:sf3} is not canonical; there is a wide latitude in choosing the representative metric tensor 
$h_\e$ on a given spherical cone-surface $N$ and the most obvious choice is not the best one for our purposes. 
We discuss this further now.

Fix a spherical cone-surface $N$ with $\ell$ conic points. We can always write any representative metric tensor $h$ in the form $e^{2\phi}F^*\overline{g}$, where $\overline{g}$ is the standard $\mbox{SO}(3)$ invariant metric
tensor on $\Sph^2$, $F$ is a diffeomorphism of $\Sph^2$ and $\phi$ is a smooth function away from the conic
points and has a logarithmic singularity (with coefficient determining the cone angle) at each cone point.
Indeed, near each $q_j$, $\phi = \beta \log s + \psi$ where $\psi$ is smooth, $s$ is the spherical distance
to $q_j$ and $2\pi(1+\beta) = \alpha$ is the cone angle there, see \cite{MW}.  

If $\sigma_\e$ is a family of spherical cone-surface structures on $N$, then the `obvious' choice is
to choose a curve of $\ell$-tuples $q(\e) = (q_1(\e), \ldots, q_\ell(\e))$ fixing a choice of locations of the 
conic points (modulo M\"obius transformations) in this deformation and then write $h_\e = e^{2\phi_\e}\overline{g}$,
where $\phi_\e$ is singular at the $q_j(\e)$. However, the differential of this family at $\e=0$ blows up
like the inverse of distance to the $q_j$ on $\Sph^2$. We prefer to find another representative for this
element by representing $h_\e = e^{2\phi_\e}F^*_\e \overline{g}$ where $F_\e$ is a diffeomorphism of $\Sph^2$ 
which maps a neighbourhood of $q_j(0)$ to a neighbourhood of $q_j(\e)$, 
for each $j$, isometrically with respect to $\overline{g}$, and where the singularities of the $\phi_\e$
remain fixed at the points $(q_1(0), \ldots, q_\ell(0))$. 

Finally, the corresponding deformation of the cone metric near the vertex $p \in M$ for which $N$ is
the link can be represented in the form
\begin{equation}
dr^2 + \sn_{\K}^2\, r\, e^{2\phi_\e}F_\e^* \overline{g},
\label{eq:sf5}
\end{equation}
where $\phi_\e$ and $F_\e$ are as described above. 

\begin{Def}
Let $\calS_\e$ be any family of singular germs (for some fixed $\Sigma$ and $\K$). We say that a family of metrics $g_\e$
representing this deformation is in standard form if it is given in the neighbourhoods of each edge and vertex by the 
formul\ae\ \eqref{eq:sf1}, \eqref{eq:sf2}, \eqref{eq:sf4} and \eqref{eq:sf5}. 
\label{def:sf}
\end{Def}

\subsection*{Infinitesimal deformations of  singular germs}
The space of infinitesimal deformations of singular germs (on a fixed graph $\Sigma$ and with curvature $\K$) 
is simply the tangent space of $\SG(\Sigma,\K)$ at some given $\calS$:
\[
T_{\calS}\SG := \left\{ \dot{g} = \left. \frac{d\,}{d\e}\right|_{\e=0} g_\e :\ g_\e \in \SG\ \mbox{for}\ 
|\e| < \e_0,\ g_0 \ \mbox{represents}\  \calS \right\}/\sim,
\]
where $\dot{g} \sim \dot{g}'$ if these two elements correspond to paths of metrics $g_\e$ and $g_\e'$ which are mutually 
isometric for each $\e$ small enough (or more generally, such that there exists a one-parameter family of diffeomorphisms
$\phi_\e$ for which $\left.\del_\e ( \phi_\e^* g_\e - g_\e')\right|_{\e=0} = 0$). This contains two distinguished subspaces: 
\[
{\mathbb V}_\calS := \{ \dot{g}\in T_\calS \SG: \alpha(e,g_\e) = \alpha(e,g_0)\ \forall\, e \in \calE\ 
\mbox{and}\ |\e| < \e_0\}/\sim
\]
consisting of derivatives of deformations which preserve all dihedral angles (the notation $\alpha(e,g)$ indicates
that the dihedral angle at $e$ is to be measured with respect to $g$), and 
\[
{\mathbb W}_\calS := {\mathbb V}_\calS \cap \{ \dot{g}\in T_\calS \SG: N(p,g_\e) = N(p,g_0)\ \forall\, p \in \calV\ 
\mbox{and}\ |\e| < \e_0\}/\sim
\]
which is the subset of $\mathbb{V}_\calS$ consisting of derivatives of deformations which preserve all spherical links (the notation $N(p,g)$ indicates the spherical cone-surface structure on $N_p$ induced by $g$).

Using Definition \ref{def:sf}, we can calculate representatives for the elements $[\dot{g}] \in T_{\calS}\SG$. Indeed, by differentiating each family of standard form deformations with respect to $\e$ we obtain a local description of the basis elements of this space. These are:
\begin{equation*}
\begin{split}
\dot{g}^{\lambda}_e & =  \cs_{\K}^2 \rho\,  f'(y)\,  dy^2, \quad \dot{g}^\tau_e  =  \sn_{\K}^2 \rho\, f'(y) \, dy d\theta, \quad
\dot{g}^\alpha_e =  \sn_{\K}^2 \rho\, d\theta^2 \\  & \mbox{and} \quad \dot{g}_p^{N}(\dot{\phi},X) = \sn_{\K}^2 r\, (
2 \dot{\phi}\, g + e^{2\phi}\calL_X \overline{g});
\end{split}
\end{equation*}
here $X$ is the infinitesimal generator of the family of diffeomorphisms $F_\e$ at $\e=0$. Thus $\dot{g}^\lambda_e$,
$\dot{g}^\tau_\e$ and $\dot{g}^\alpha_e$ correspond to changing the length, twist parameter and cone angle along an edge $e$,
while $\dot{g}_p^{N}(\dot{\phi},X)$ represents the effect of changing the spherical cone-surface structure on the link of $p$.
(The coefficient of the logarithmic singularity in $\dot{\phi}$ encodes the change of cone angles while $X$ represents the changes 
of position of the conic points.)  

As already noted, if $[\dot{g}]$ preserves all dihedral angles, then it can be written as a sum of noninteracting deformations
localized along each edge and vertex. In other words,
\begin{equation*}
{\mathbb W}_{\calS} \ni \dot{g} \sim \sum_{e \in \calE} \left(a_e \dot{g}^{\lambda}_e 
+ b_e \dot{g}^{\tau}_e \right)\label{eq:bW}
\end{equation*}
and
\begin{equation*}
{\mathbb V}_{\calS} \ni \dot{g} \sim \sum_{e \in \calE} \left(a_e \dot{g}^{\lambda}_e 
+ b_e \dot{g}^{\tau}_e \right) + \sum_{p \in \calV} \dot{g}_p^N(\dot{\phi},X)
\label{eq:bV}
\end{equation*}
where each $(\dot{\phi},X)$ represents an infinitesimal change of spherical cone metric in standard form on $N_p$. 
In this case, because the angles are fixed, $\dot{\phi}$ is smooth on $\Sph^2$. The infinitesimal deformations which
correspond to a variation of dihedral angles also include terms of the form $c_e \dot{g}^{\alpha}_e$ as well as 
$\dot{g}_p^N(\dot{\phi},X)$ where $\dot{\phi}$ has a logarithmic singularity. 

\begin{Prop} \label{prop:stdl2}
Every infinitesimal deformation $\dot{g} \in T_{\calS} \SG$ in standard form is represented by a symmetric $2$-tensor in 
the neighbourhood $\calU$ of $\Sigma$ which is polyhomogeneous and lies in $L^2(\calU; g)$ (where $g$ is the cone metric representing $\calS$).
However, the covariant derivative $\n \dot{g}$ lies in $L^2$ if and only if it corresponds to a deformation of 
metrics which leaves the dihedral angles invariant.
\end{Prop}

We refer to \S \ref{sec:phg} for the definition of polyhomogeneity. The proof follows from the observation that while $\sn_\K^2\rho
 \, d\theta^2$ lies in $L^2$ near an edge, its covariant derivative equals $-\cs_\K \rho\, \sn_\K \rho\, 
d\theta \otimes (d\rho \otimes d\theta + d\theta \otimes d\rho)$ and this is not in $L^2$ near $\rho=0$. 

This result is the key to our proof of the infinitesimal Stoker conjecture for cone-manifolds since it shows how to
distinguish the angle-preserving deformations from all others by a simple analytic criterion. This explains the importance
of putting infinitesimal deformations into standard form. 

We may now finally rephrase Theorem \ref{thm:mainthm} and provide a succinct statement of the infinitesimal Stoker conjecture.

\begin{Thm} Let $(M,g_\e)$ be a family of cone-manifolds with fixed singular locus $\Sigma$ and curvature $\K$. Let $\dot{g} 
\in T_{\calS}\SG$ be the derivative at $\e=0$ of the corresponding family of singular germs $G(g_\e)$, where $G(g_0) = \calS$. 
Suppose that the infinitesimal deformation preserves the dihedral angles, i.e.~$\dot{g} \in \mathbb{V}_\calS$.
If $\K =  0$, then $\dot{g} \in  \mathbb{W}_\calS$; if $\K =  -1$, then $\dot{g} \sim 0$ on all of $M$.
\label{th:sc}
\end{Thm}

Proposition~\ref{prop:stdl2} states that all infinitesimal deformations which preserve dihedral angles have polyhomogeneous 
(at $\Sigma$) representatives $\dot{g}$ such that both $\dot{g}$ and $\nabla \dot{g}$ lie in $L^2$; the proof of 
Theorem~\ref{th:sc} shows that in the hyperbolic case, any global infinitesimal deformation which has this property must be trivial,
and hence the corresponding family of metrics $g_\e$ must all be isometric to first order. Another way to state this is that any 
nontrivial deformation of a hyperbolic cone-manifold $(M,g)$ must necessarily change the cone angles, and that no 
representative $\dot{g}$ for the corresponding infinitesimal deformation can lie in $L^2$ along with its covariant derivative. 
The proof in the Euclidean case involves showing that the infinitesimal deformation is a harmonic $1$-form with values
in $T^*M$, and then interpreting this conclusion geometrically.

\section{Analytic deformation theory}
In this section we recall a framework from geometric analysis for studying the deformation theory of cone manifolds. 
The monograph \cite{Besse} contains a comprehensive (albeit now somewhat outdated) review of deformation theory 
of Einstein metrics; we refer also to \cite{Biquard}, which provides a closer guide for the treatment here.

\subsection*{The Einstein equation and the Bianchi gauge}
If $M$ is a manifold of dimension $n$ with a metric $g$ of constant sectional curvature $\K$, then $g$ 
satisfies the Einstein equation $E(g)=0$, where $E(g):=\Ric(g) - (n-1)\K g$; when $\dim M = 2$ or $3$, the converse is true as well. 
Hence in these low dimensions we may apply formalism developed for the Einstein equation to study cone-manifolds.

The Einstein equation is diffeomorphism invariant, so for any metric tensor which satisfies this equation there is
an infinite dimensional family of nearby metrics which also do, namely anything of the form $F^*g$ where 
$F$ is a diffeomorphism  of the underlying space. A one-parameter family $g_\e$ of 
deformations of a given Einstein metric $g = g_0$ is called trivial if there exists a one-parameter family of 
diffeomorphisms $F_\e$ with $F_0 = \mbox{Id}$ and such that $g_\e = F_\e^*g$. 

If $g_\e$ is a one-parameter family of Einstein metrics (i.e.~$E(g_\e) =0$ for all $\e$) then the derivative $k = \frac{d}{d\e}g_\e|_{\e=0}$ satisfies $DE_{g_0}(k)=0$; any such symmetric $2$-tensor is called an infinitesimal Einstein deformation of $g_0$. 
To obtain a clean formula for the linearization $DE_g$ of the Einstein equation, let us introduce the 
Bianchi operator $B^g$, which carries symmetric $2$-tensors to $1$-forms, 
\[
B^g(k) = \delta^g k + \frac12 d\, \tr^g \, k.
\]
Here $\tr^g \, k = k^i_{\,i} = g^{ij}k_{ji}$ designs the trace of $k$ with respect to $g$, and $(\delta^g k)_i = -(\n^g k)^j_{\,ji}$ is the operator $\na$ applied to symmetric $2$-tensors. 
Note that $B^g(g) = 0$ for any metric $g$, and more interestingly, $B^g(\Ric(g)) = 0$ as well, which
is simply the contracted second Bianchi identity. Now define the operator $L^g$ which acts on a symmetric 
$2$-tensor $k$ by
\begin{equation*}\label{eq:Lg1}
L^g k  = \nabla^* \nabla k - 2 \RO k \ +  \Ric \circ k + k \circ \Ric - 
2 \, (n-1)\K \, k;
\end{equation*}
where 
\[
(\RO k)_{ij} = R_{ipjq}\, k^{pq}, \qquad (\Ric \circ k)_{ij} = 
\Ric_i^{\ p}\, k_{pj}, \qquad (k \circ \Ric)_{ij} = k_i^{\ p} \,\Ric_{pj},
\]
and all curvatures and covariant derivatives are computed relative to $g$.  In particular, when $\Ric(g) = (n-1){\K} g$, then 
\[
L^g = \nabla^* \nabla - 2 \RO;
\]
if $g$ has constant sectional curvature $\K$ then the action of $\RO$ is scalar on the pure-trace and trace-free parts of $k$, and
\begin{equation}\label{eq:Lg2}
L^g k = \na \n k +2\K (k- \tr^g(k)g).
\end{equation}
In any case, the linearized Einstein operator has the expression
\begin{equation*}
DE_g = \frac{1}{2}L^g -  (\delta^g)^*B^g;
\end{equation*}
here $((\delta^g)^* \omega)_{ij} = \frac12 ((\n^g \omega)_{ij}+(\n^g \omega)_{ji})$ is the adjoint of $\delta^g$ defined above.
Clearly, any symmetric $2$-tensor $k$ of the form 
\[
k = \frac{d\,}{d\e}F_\e^* g |_{\e=0} = \calL_X g = 2 (\delta^g)^* \omega 
\]
satisfies $DE_g(k) = 0$. Here $F_\e$ is a family of diffeomorphisms with $F_0 = \mbox{Id}$,
$X$ is the generating vector field and $\omega$ its dual $1$-form. 
\begin{Def}
Let $g$ be an Einstein metric. An infinitesimal Einstein deformation of $g$ is any symmetric $2$-tensor $k$ which 
satisfies $DE_g(k) = 0$. It is called trivial if $k = 2 (\delta^g)^* \omega$ for some $1$-form $\omega$.
\end{Def}

The Bianchi operator provides a gauge in which the Einstein operator is elliptic, and in our case positive; let us recall briefly how 
this goes in the compact smooth case. Let $g$ be Einstein and set
\[
N^g(k) := \Ric(g+k) - (n-1)\K (g+k) + (\delta^{g+k})^*B^g(k).
\]
Then $N^g(k) = 0$ if $g+k$ is Einstein and $B^g(k) = 0$, but conversely, as proved in \cite{Biquard}, 
\begin{Prop}
If $M$ is a compact smooth manifold and $N^g(k) = 0$, then $g+k$ is Einstein and in Bianchi gauge, i.e.\ $B^g(k) = 0$, 
provided $\Ric(g+k)$ is negative. If $\Ric(g+k) \equiv 0$, then $B^g(k)$ is parallel. 
The same conclusions are true when $M$ is a cone-manifold provided the integrations by parts 
in the proof indicated below can be justified.
\end{Prop}
In the smooth case, the proof follows immediately from the Weitzenb\"ock formula
\[
2B^{g+k}N^g(k) = \left((\nabla^{g+k})^*\nabla^{g+k} - \Ric(g+k)\right)(B^g(k))
\]
and the integration by parts
\[
\begin{split}
\int_M \langle \left((\nabla^{g+k})^*\nabla^{g+k} - \Ric(g+k)\right)(B^g(k)), B^g(k) \rangle  \\ 
= \int_M |\nabla^{g+k}B^g(k)|^2 - \langle \Ric(g+k)B^g(k),B^g(k)\rangle.
\end{split}
\]
For cone-manifolds, one must verify that this last step introduces no extra boundary terms. 

\subsection*{Gauging an infinitesimal deformation}\label{sec:gauge}

We now briefly sketch how Theorem \ref{thm:mainthm} is proved below. Let $(M,g)$ be a closed cone-manifold and 
$\dot{g}$ an angle-preserving infinitesimal deformation of this cone-manifold structure which is in standard form. 
In particular, according to Proposition \ref{prop:stdl2} $\dot{g}$ is polyhomogeneous along the singular locus, with
$\dot{g}$, $\n \dot{g} \in L^2$, and in addition $DE_g(\dot{g}) = 0$. Our goal is to prove that in the hyperbolic 
case, $\dot{g}$ is a trivial Einstein deformation, i.e.\ that $\dot{g} = 2\delta^* \omega$ for some $1$-form $\omega$ 
with reasonable regularity properties at $\Sigma$ (in the Euclidean case we prove that $\dot{g}$ is equivalent to 
a harmonic $T^*M$-valued $1$-form). The main step is to put $\dot{g}$ in Bianchi gauge. Thus we 
seek a $1$-form $\eta$ such that
\[
B^g (\dot{g} - 2 \delta^* \eta) = 0,
\]
or equivalently, using the classical $1$-form Weitzenb\"ock identity $\delta d + d \delta = \na \n +\Ric$, we must find $\eta$ as a solution to the equation
\begin{equation}
P^g \eta = B^g\dot{g},
\label{eq:tmeq}
\end{equation}
where $P^g$ is the operator 
\begin{equation*}
P^g = 2 B^g \delta^* = (\nabla^g)^* \nabla^g - \Ric(g)
\end{equation*}
acting on $1$-forms. 

The main result which will allow us to prove the infinitesimal rigidity theorem is the following:
\begin{Thm}\label{thm:gauge}
Let $M$ be a compact connected three-dimensional cone-manifold which is either hyperbolic or Euclidean, and suppose that all 
of its cone angles are smaller than $2\pi$. If it is Euclidean, then suppose too that it has at least one singular point of codimension $3$.

If $f \in L^2(M;T^*M)$ and is polyhomogeneous, then there exists a unique polyhomogeneous solution $\eta$ to the equation $P^g\eta = f$ 
such that $\eta$, $\n \eta$, $\n d \eta$, and $d \delta \eta$ all lie in $L^2$. 
\end{Thm}
There is a version of this theorem which does not assume polyhomogeneity, but the version here is simpler to work with
and is all we need anyway.  By construction, $h = \dot{g} - 2 \delta^* \eta$ lies in the kernel of the operator $L^g$. 
This gauged deformation is polyhomogeneous, and by examining the specific structure of the terms in its expansion, we can
justify the integration by parts which shows that $h$ either vanishes (hyperbolic case) or is a harmonic $T^*M$-valued 
$1$-form (Euclidean case).

\section{Some elliptic theory on stratified spaces}\label{ellcm}
In the preceding sections we introduced the formalism through which the infinitesimal rigidity problem can be transformed 
into a set of purely analytic questions concerning existence and behaviour of solutions of certain elliptic operators 
on conifolds. We now discuss the analytic techniques needed to study these questions. Although much of this can
be carried out for general elliptic `iterated edge operators', for brevity we only describe the main results for 
generalized Laplacians on tensor bundles, i.e.\ operators of the form $\nabla^* \nabla + \calR$ where $\calR$ 
is an endomorphism built out of the curvature.  This includes the specific operators $L^g$ and $P^g$ on cone-surfaces and cone-manifolds. 

We begin by discussing general iterated edge operators and reviewing the notions of conormal and polyhomogeneous 
regularity. This leads to a more detailed description of the three `low depth' cases of interest here: generalized Laplacians
for metrics which are conic, have incomplete edges, or which are cones with links equal to spaces with isolated conic singularities.
For simplicity, we refer to singularities of this last type as restricted depth $2$ (since general depth $2$ singularities
are cone bundles over spaces with isolated singularities). In each case we state the basic mapping properties and the 
regularity of solutions. This is required to understand the self-adjoint extensions of these operators, which along
with the regularity theory is an important tool later in this paper. 

The discussion in the remainder of this section is phrased for operators acting on functions rather than sections of
bundles. There is no loss of generality in doing this, and all these results generalize (usually trivially) to systems. 

\subsection*{Incomplete iterated edge operators}
Just as iterated edge spaces are defined inductively, so too are the natural classes of differential operators on these 
spaces. Let $M$ be an iterated edge space, and $S$ a stratum of depth $d$; thus, any point of 
$S$ has a neighbourhood diffeomorphic to the product $\calU \times C_{0,1}(N)$ where $\calU$ is an open ball 
in $\RR^m$, $m = \dim S$, and  the link $N$ is an iterated edge space of depth strictly less than $d$. Let $r$ 
be the radial variable in $C_{0,1}(N)$, $y$ a coordinate system in $\calU$, and $z$ a `generic' variable in $N$. 
If $\calL$ is a generalized Laplacian on $M$ associated to an incomplete iterated edge metric $g$, then there is an 
induced generalized Laplacian $\calL_N$ on $N$, so that in this neighbourhood, and with respect to appropriate
trivializations of the bundles between which these operators act, 
\begin{equation}
\calL = -\del_r^2 - \frac{n-m-1}{r}\del_r + \frac{1}{r^2}\calL_N + \Delta_S + \frac{1}{r}E;
\label{eq:mgL}
\end{equation}
here 
\[
E =  A_{1,0}(r,y)\, r\del_r + \sum_{|\alpha| \leq 2} A_{0,\alpha}(r,y)\, \del_y^\alpha
\]
where each $A_{0,\alpha}$ is an incomplete iterated edge operator (the precise definition of which is given just below) of 
order $2-|\alpha|$ on $N$ which is smooth down to $r=0$ and  $A_{1,0}$ is an operator of order $0$ (i.e.\ it acts by 
matrix multiplication). The important point is not the exact expression of $E$ but the fact that it contains only 
`lower order terms' in a sense that will be made clear soon. 
If $d>1$, then $\calL_N$  has a similar expression. For example, when $d=2$, then using `edge coordinates' $(s,w,\theta)$ 
on $N$, where $\theta$ is a variable on the smooth link $N'$ for the cone-bundle structure of $N$, we can write 
\begin{equation}
\calL = -\del_r^2 - \frac{n-m-1}{r}\del_r + \frac{1}{r^2}(-\del_s^2 - \frac{\ell}{s}\del_s + 
\frac{1}{s^2}\calL_{N'} + \Delta_w) + \Delta_S + \frac{1}{r}E_2,
\label{eq:mgL2}
\end{equation}
where once again $E_2$ contains the various lower order terms and $\ell=\dim N'$.

The key point is that when $d=1$, the operators $r^2\calL$ and $r\calL r$ are sums of smooth multiples of products 
of the vector fields $r\del_r, r\del_y, \del_z$, while if $d=2$, then $r^2 s^2 \calL$ is a similar combination of the vector fields 
$r s \del_r, s\del_s, r s \del_y, s \del_w, \del_\theta$. We shall denote the span over $\calC^\infty$ of these 
the \emph{edge} and \emph{iterated edge vector fields}, $\calV_\mathrm{e}$ and $\calV_{\ie}$, respectively; if $d=1$ and the singular 
stratum consists of isolated points, then $\calV_\mathrm{e}$ reduces to the space of $b$-vector fields $\calV_b$, see \cite{Melrose}.  
In general, an \emph{iterated edge operator} in one of these three cases is any differential operator which is a locally finite sum 
of products of elements of $\calV_b$, $\calV_{\mathrm{e}}$ or $\calV_{\ie}$. As above, we shall actually be dealing with
operators of the form $r^{-m}A$ or $(r s)^{-m}A$ where $A$ is an iterated edge operator of order $m$,
and we call these conic, incomplete edge or incomplete iterated edge operators. 

Associated to these three spaces of vector fields and the associated operators are the $L^2$-based weighted Sobolev spaces:
\[
\begin{array}{rcl}
H^\ell_*(M) & = & \{u: V_1 \cdots V_j u \in L^2\ \forall\, j \leq \ell,\ V_i \in \calV_*\}, \\
r^\delta s^{\delta'}H^\ell_*(M) & = & \{u = r^\delta s^{\delta'} v: v \in H^\ell_*(M)\}, \qquad
* = b, \mathrm{e}\  \mbox{or}\ \ie.
\end{array}
\]
We also define $r^\delta s^{\delta'} H^\infty_*(M)$ to be the intersection of the above spaces for all integers $\ell$.
Of course, when $* \neq \ie$, the $s^{\delta'}$ factor is absent. To avoid redundancy, in the rest of this section we
follow a similar convention by writing various formul\ae\ with both $r$ and $s$ factors present, with the
understanding that $s$ should  be omitted unless we are discussing operators on depth $2$ spaces.

Directly from these definitions, if $\calL$ is a generalized Laplacian on an iterated edge space, then
\begin{equation}
\calL: r^{\delta} s^{\delta'}H^{\ell+2}_*(M) \longrightarrow r^{\delta-2}s^{\delta'-2}H^\ell_*(M)
\label{eq:mapedge}
\end{equation}
is bounded for any real numbers $\delta, \delta'$.  The more interesting problem is to determine whether these mappings 
are Fredholm, and whether solutions of $\calL u = 0$ have special regularity properties. It turns out that the answers to 
these questions are sensitive to the values of the weight parameters. 

\subsection*{Domains of closed extensions of conic operators}
The Fredholm properties for the incomplete iterated edge operator $\calL$ acting between weighted Sobolev spaces with 
{\it different} weights is actually not what is needed below; we are more interested in the unbounded operator
\begin{equation*}
\calL: L^2(M) \longrightarrow L^2(M).
\end{equation*}
More specifically, we consider closed extensions of the operator $\calL$, initially acting on the core domain 
$\calC^\infty_0(M \setminus \Sigma)$. In other words, we wish to choose a domain $\calD$ with 
$\calC^\infty_0(M \setminus \Sigma) 
\subset \calD \subset L^2(M)$ such that the graph of $\calL$ over $\calD$, $\{(u,\calL u): u \in \calD\}$, is a closed 
subspace; we are most interested in domains $\calD$ such that the corresponding unbounded Hilbert space operator 
$(\calL, \calD)$ is self-adjoint. We refer to \cite{Rudin} for a discussion of this general theory, and to \cite{Gil-Mend} or 
\cite{Lesch} for its application in the conic setting. 

There always exist two canonical closed extensions of the differential operator $\calL$:
\begin{itemize}
\item[a)] The maximal domain of $\calL$ is the subspace
\[
\calD_{\max}(\calL) = \{u \in L^2(M): \calL u \in L^2(M) \};
\]
where $\calL u$ is defined in the sense of distributions. This is the largest possible subset of $L^2$ on which $\calL$ can be defined (while preserving the fact that its adjoint domain contains $\calC^\infty_0$, see below). 
\item[b)] The minimal domain of $\calL$ is the subspace 
\begin{multline*}
\calD_{\min}(\calL) = \{u \in L^2(M): \exists\, u_j \in \calC^\infty_0(M\setminus \Sigma)\ \mbox{such that}\ \\
||u_j - u||_{L^2} \to 0,\ \{\calL u_j\}\ \mbox{Cauchy in}\ L^2 \}.
\end{multline*}
In other words, $u \in \calD_{\min}$ if $\calL u = f$ distributionally and $(u,f)$ lies in the closure in $L^2 \times L^2$ 
of the graph of $\calL$ over $\calC^\infty_0$. 
\end{itemize}
A straightforward argument using cutoff functions shows that $r^2 s^2 H^2_{\ie}(M) \subseteq \calD_{\min}$, and
it is always the case that $\calD_{\min} \subseteq \calD_{\max}$. 

If $(\calL, \calD)$ is any closed extension of $\calL$, then the Hilbert space adjoint $(\calL, \calD^*)$ is defined as follows:
the adjoint domain is given by
\[
\calD^* = \{v \in L^2: |\langle \calL u,v\rangle| \leq C ||u||_{L^2}\ \mbox{for all}\ u \in \calD\}.
\]
(We are using the inner product $\langle u, v \rangle = \int_M g(u,v)\, dV_g$, so that $\calL$ is symmetric, i.e.\ 
$\langle \calL u, v \rangle = \langle u, \calL v \rangle$ for all $u,v \in \calC^\infty_0$.)   Then, for $v \in \calD^*$,
define $\calL v$ by the Riesz representation theorem: $\calL v = f$ is the unique element in $L^2$ such that 
$\langle \calL u, v \rangle = \langle u, f \rangle$ for all $u \in \calD$. 

It is straightforward to check that $\calD_{\max} = \calD_{\min}^*$ and $\calD_{\min} = \calD_{\max}^*$, and more generally,
that if $\calD_1 \subseteq \calD_2$, then $\calD_2^* \subseteq \calD_1^*$. 
An extension $(\calL,\calD)$ is called self-adjoint if $\calD^* = \calD$; in particular, if $\calD_{\min} = \calD_{\max} := 
\calD$, then $\calL$ is called \emph{essentially} self-adjoint and $(\calL, \calD)^* = (\calL, \calD)$ is the unique self-adjoint 
extension of $\calL$. For convenience, we often write $\calL_{\calD}$ for $(\calL,\calD)$, and in particular $\calL_{\max}$ 
or $\calL_{\min}$ when $\calD = \calD_{\max}$ or $\calD_{\min}$.  

If $\calL$ has the special form $\calL = A^* A + \calR$ where $\calR$ is bounded above and below, then there is a natural 
domain on which $\calL$ is self-adjoint, called the Friedrichs extension. It is defined by
\[
\calD_{\mathrm{Fr}}= \{u \in \calD_{\max}(\calL): \exists \, u_j \in \calC^\infty_0\ \mbox{with}\ ||u_j - u||_{L^2} \to 0,
\ Au_j\ \mbox{Cauchy}\, \}.
\]
By definition, $\calL_{\Fr} = (A_{\min})^* A_{\min} + \calR= (A^*)_{\max}A_{\min} + \calR$.

To illustrate how this arises, consider the case where $A = \nabla$, acting between certain tensor bundles.
We begin with the $L^2$ Stokes' theorem on conifolds, which is a special case of a result in Cheeger's fundamental 
paper \cite{Cheeger}; see also \cite[Appendix]{HK}, \cite{HM}, and \cite[\S 1.4]{MontcouqThese} for  more details.
\begin{Prop} \label{thm:stokes} 
Let $M$ be a cone-surface or three-dimensional conifold with singular locus $\Sigma$. Suppose that $u \in L^2(M,T^{(r,s)}M)$ 
and $v \in L^2(M,T^{(r+1,s)}M)$ satisfy $u, \n u, v, \na v \in L^2(M)$. Then 
\[
\langle \n u,v \rangle = \langle u,\na v\rangle. 
\]
On the other hand, if $u$ and $v$ are differential forms such that $u$, $d u$, $v$,  $\delta v \in L^2$, then
\[
\langle du , v \rangle = \langle  u, \delta v \rangle. 
\]
More succinctly, $\n_{\min} = \n_{\max}$, $d_{\min} = d_{\max}$, and $\delta_{\min} = \delta_{\max}$. 
\end{Prop}
\begin{proof} (Sketch) Let $\calU_\e$ be the $\e$ neighbourhood of $\Sigma$, and set $M_\e = M\setminus \calU_\e$. 
Then for $u, v \in \calC^\infty(M\setminus\Sigma)\cap L^2(M \setminus \Sigma)$, 
\[
\int_{M_\e} \left(g(u,\na v) - g(\n u,v) \right) = \int_{\d M_\e} g(u,i_{n} v),
\]
where $n$ is the unit normal to $\d M_\e$ and $i_n(v)(.) = v(n,\cdot )$. The left side converges to $\<u,\na v\> - \<\n u,v\>$ 
as $\e \searrow 0$, so the main point is to show that the limit of the right side vanishes. The Cauchy-Schwarz inequality gives
\[
\left|\int_{\d M_\e} g(u,i_{n}v) \right| \leq \left(\int_{\d M_\e} |u|^2\right)^{1/2} \left(\int_{\d M_\e} 
|i_{n} v|^2\right)^{1/2}.
\]
The fact that $\n u$ is in $L^2$ implies that $u$ decays near the singular locus. More precisely, one can show that $\int_{\d M_\e} |u|^2 = \calO(\e \log \e)$. In the same spirit, since $v$ and hence $i_n(v)$ lie in $L^2$, they cannot grow to fast near $\Sigma$: there exists a sequence $\e_j \to 0$ such that $\int_{\d M_{\e_j}} |i_{n} v|^2 = 
o((\e_j \log \e_j )^{-1})$. When $u$ and $v$ are only $L^2$, then one must use smooth approximations of $u$ and $v$. In any 
case, it follows that the limit of the right side vanishes as $\e \to 0$. The proof for $d$ and $\delta$ is similar.
\end{proof}

As an immediate consequence, we have the
\begin{Cor}
If $\calL = \nabla^* \nabla$, acting on any tensor bundle, then its Friedrichs extension satisfies
\[
(\na \circ \n)_{\Fr} = \na_{\max}\circ \n_{\min} = \na_{\min} \circ \n_{\max} = \na_{\max} \circ \n_{\max} = 
\na_{\min} \circ \n_{\min}; 
\]
in particular,
\[
\calD_{\Fr}(\na \circ \n +\calR) = \calD_{\max}(\nabla) \cap \calD_{\max}(\na \circ \n) = 
\{u \in L^2 \ |\ \n u \in L^2, \na \n u \in L^2\}.
\]
\end{Cor}

\subsection*{Conormality and polyhomogeneity}\label{sec:phg}
It is very convenient to be able to work with functions which are not just in some weighted iterated edge Sobolev space, 
but are more regular. Experience dictates that the right class of `smooth functions' on a conifold consists
of functions which are polyhomogeneous; this regularity condition arises naturally, for example, for solutions
of elliptic iterated edge equations. Roughly speaking, polyhomogeneous functions are those which
admit expansions near each singular locus with tangentially smooth coefficients, and which have `product-type'
behaviour near the higher depth singularities. 

Suppose first that $M$ has only isolated conic singularities, so that $\calV_b$ is the space of structure
vector fields. We say that $u$ is conormal of ($L^\infty$) weight $\delta$, $u \in \calA^\delta_{L^\infty}(M)$, 
if $|(r\del_r)^j \del_z^\alpha u | \leq C r^\delta$ for all $j, \alpha$ and for some fixed $\delta \in \RR$ which 
is independent of $j$ and $\alpha$. We could equally well have altered this definition to require that $u$ and 
all of its $b$-derivatives lie in some fixed weighted $L^2$ space, $r^\delta L^2(M, r^{-1} dr dz)$; if this condition
is satisfied, we write $u \in \calA^{\delta}_{L^2}(M)$.  These $L^2$-based conormal spaces are slightly different from 
the $L^\infty$-based ones defined first, but it is not hard to check that for each $\delta \in \RR$ and any $\e > 0$, 
$\calA^{\delta+\epsilon}_{L^\infty}(M) \subset \calA^{\delta}_{L^2}(M) \subset \calA^{\delta - \epsilon}_{L^\infty}(M)$, hence the 
intersection of these spaces over all $\delta$ are the same, and similarly, their union over all $\delta$ are the same. 
In various proofs below we shall 
find it simpler to work with the $L^2$-based spaces instead, but the $L^\infty$-based ones are slightly simpler conceptually,
so we presented their definition first. In any case, we say that conormal functions have {\it stable regularity} with 
respect to $\calV_b(M)$.  

Now suppose that $M$ has a simple edge singularity. Then $u$ is conormal at the edge if it has stable 
regularity in exactly the same sense with respect to the $b$-vector fields on $M$; note that this space of
vector fields is generated by $r\del_r, \del_y, \del_z$, which is {\it not} the same as $\calV_\mathrm{e}$, for then
we would only be requiring that $u$ have stable regularity upon taking derivatives with respect to $r\del_r, r \del_y$ 
and $\del_z$, so we would obtain only much weaker information on regularity in the edge ($y$) direction.
Finally, if $M$ has an isolated singular point of depth $2$, then in terms of coordinates $(r,y,s,w,\theta)$ as above,
$u$ is conormal if it has stable regularity with respect to the space of vector fields generated by $r \del_r, s\del_s, 
\del_y, \del_w, \del_\theta$; again note that this is much stronger than requiring stable regularity with respect to $\calV_{\ie}$.

The difference between conormality and regularity with respect to $\calV_\mathrm{e}$ or $\calV_{\ie}$ is one of the central points
of the theory of iterated edge operators. On the one hand, there is a very general regularity result which holds for solutions of 
arbitrary elliptic iterated edge operators; for simplicity we state this only for generalized Laplacians on depth $2$ spaces,
so as not to have to define ellipticity in this general setting or describe notation for weight functions in spaces of general depth.

\begin{Prop} Let $\calL$ be a generalized Laplacian for an iterated edge metric of depth $2$ on a space $M$. Suppose that
$f\in r^{\delta-2} s^{\delta'-2} H^\ell_\ie(M)$ for some $\ell \geq 0$ and weights $\delta, \delta' \in \RR$. 
If $u \in r^\delta s^{\delta'}L^2(M)$ solves $\calL u = f$, then $u \in r^\delta s^{\delta'}H^{\ell+2}_\ie(M)$; 
in particular, if $f = 0$, then $u \in r^\delta s^{\delta'}H^\ell_{\ie}$ for all $\ell \geq 0$ -- in other words, 
$u$ has stable regularity with respect to $\calV_{\ie}$. 
\label{pr:nantes2}
\end{Prop}
A close examination of this statement shows that it reduces by scaling arguments to standard local elliptic estimates. 
It can also be proved using the so-called uniform pseudodifferential calculus on the interior of $M$, see \cite[\S 4]{ALMP} for 
a more careful description of this. On the other hand, it is not true that an arbitrary solution to $\calL u = 0$ is conormal, 
and it is an important problem to give criteria to ensure that solutions do in fact enjoy conormal regularity.

Functions which satisfy these conormality conditions include any monomial of the form $r^\delta a(y,z)$ or
$r^\delta s^{\delta'} a(y,w,\theta)$ in the depth $1$ and $2$ cases, respectively, where the coefficients
are smooth in all other variables; we can also take infinite asymptotic sums of such monomials. However,
many other functions satisfy the conormal estimates, including arbitrary powers of $|\log r|$, etc.
A more tractable subclass consists of the conormal functions $u$ which have asymptotic expansions in
terms of the simple monomials above; for technical reasons it is necessary to only allow positive integer powers 
of the logs of the defining functions. Thus in the depth $1$ case, $u$ is polyhomogeneous, 
$u\in \calA_{\phg}(M)$, if
\[
u \sim \sum_{\mathrm{Re}\, \gamma_j \to \infty} \sum_{\ell = 0}^{N_j}
r^{\gamma_j} (\log r)^\ell a_{j,\ell}(y,z),
\]
whereas in the depth $2$ case, $u \in \calA_{\phg}(M)$ if it has an expansion of this form near each
of the boundaries $r \to 0$, $s \geq \e$ and $s \to 0$, $r \geq \e$, while near the corner
$r = s = 0$ it has a double expansion
\[
u \sim \sum_{\genfrac{}{}{0pt}{}{\mathrm{Re}\, \gamma_j \to \infty}{\mathrm{Re}\, \eta_i \to \infty}}
\sum_{\ell=0}^{N_j} \sum_{k=0}^{N'_i}  r^{\gamma_j} s^{\eta_i} (\log r)^\ell (\log s)^k a_{i,j,k,\ell}(y,w,\theta),
\]
where again all coefficients are smooth. 

\subsection*{Conic operators}
A space $(M^n,g)$ with the simplest iterated edge structure is one with isolated conic singularities. Each cone point $p\in M$ 
has a neighbourhood $\calV$ diffeomorphic  to a cone $C_{0,1}(N)$ where $N$ is a compact, smooth $(n-1)$-dimensional 
manifold. A generalized Laplacian $\calL$ on $M$ can be written in $\calV$ as in (\ref{eq:mgL}), but since the singular
stratum $S$ reduces to the point $p$, there are no $y$ variables, and $N$ is smooth. 

Now define 
\[
A = r \calL r = -(r\del_r)^2 - nr \del_r + \calL_N - (n-1) + r E',
\]
where $E'=r^{-1}Er$ is again a smooth combination of multiples of $r\del_r$ and $\del_z$. This is a $b$-operator 
in the sense of \cite{Melrose}, cf.\ also \cite{Ma-edge}; we have put the two factors of $r$ on both sides of $\calL$
to make it symmetric on $L^2$, as this will be useful below.

A number $\gamma \in \CC$ is said to be an indicial root of $A$ if there exists $\phi \in \calC^\infty(N)$ such 
that $A(r^\gamma \phi(z)) = \calO(r^{\gamma+1})$. Writing this out, we see that
\begin{equation}
A(r^\gamma \phi) = r^\gamma (\calL_N - \gamma^2 - n\gamma - n+1) \phi + \calO(r^{\gamma+1}),
\label{eq:nantes}
\end{equation}
hence $\gamma$ is an indicial root for $A$ if and only if $\gamma^2 + n\gamma + n-1$ is an eigenvalue of $\calL_N$ 
and $\phi$ is the corresponding eigenfunction. We define the indicial roots of $\calL$ similarly, so $\gamma$ 
is an indicial root for $\calL$ if and only if $\gamma - 1$ is an indicial root for $A$.

We have already noted that 
\begin{equation}
\calL: r^{\delta}H^{\ell+2}_b(M) \longrightarrow r^{\delta-2} H^\ell_b(M)
\label{eq:mapwss}
\end{equation}
is bounded for any $\delta, \ell$. The first basic result is that (\ref{eq:mapwss}) is Fredholm whenever $\delta$ is 
not an `indicial weight'. More precisely, suppose that $\gamma$ is an indicial root of $\calL$, and set 
$\delta(\gamma) = \gamma + n/2$. Then $r^\gamma$ lies in $r^{\delta(\gamma) - \e} L^2((0,1);r^{n-1}dr)$ 
for every $\e > 0$, but not when $\e=0$; in other words, $r^\gamma$ `just fails' to lie in $r^{\delta(\gamma)}L^2$. 
There is a sequence of cutoffs $u_j$ of $r^\gamma \phi(z)$ (where $\phi$ is associated to $\gamma$ as
above) with disjoint supports such that
\[
||u_j||_{r^{\delta(\gamma)}L^2} = 1, \qquad ||\calL u_j||_{r^{\delta(\gamma)-2}L^2} \to 0.
\]
This proves that \eqref{eq:mapwss} does not have closed range when $\delta = \delta(\gamma)$, but in fact, this is the only 
obstruction to Fredholmness: 
\begin{Prop}
Let $\calL$ be a generalized Laplacian on a compact space $(M,g)$ with isolated conic singularities.  Then (\ref{eq:mapwss}) 
is Fredholm provided $\delta \neq \delta(\gamma)$ for some indicial root $\gamma$ of $\calL$. 
\label{pr:Fredconic}
\end{Prop}
This is Proposition 4.4 in \cite{Ma-edge}, applied to the operator $r^2 \calL$. (This is a much older result, of course, 
dating back to the work of Kondratiev in the '60's, and also follows directly from various well-known papers including
those of Cheeger, Melrose and Lockhart-McOwen in the '70's and '80's; for this simple Fredholm statement alone, the
proof is quite elementary, using only separation of variables.)

The indicial roots play a fundamental role in the regularity theory of solutions too:
\begin{Prop}
Let $(M,g)$ and $\calL$ be as above, and $\delta \in \R$ a weight such that $\delta \neq \delta(\gamma)$ for any indicial root 
$\gamma$ of $\calL$. If $f \in r^{\delta-2}H^{\ell}_b(M)$ for any $\ell \in \NN$  and $u \in r^\delta L^2$ is a solution of 
$\calL u = f$, then $u \in r^{\delta} H^{\ell+2}_b(M)$. In particular, taking the intersection over all $\ell$,  if $f \in \calA(M)$, 
then $u \in \calA(M)$, i.e.\ if $f$ is conormal, then so is $u$.  Moreover, if $f \in \calA_{\phg} \cap r^{\delta-2}L^2$ (e.g.\ if
$f \equiv 0$) and $u \in r^\delta L^2$, then $u \in \calA_{\phg} \cap r^\delta H^\ell_b$ for any $\ell \geq 0$.  The exponents 
in the polyhomogeneous expansion of $u$ are all of the form $\gamma + \ell$ where $\gamma$ either is an indicial root 
of $\calL$ or else $\gamma = \gamma' + 2$ where $\gamma'$ appears in the polyhomogeneous expansion for $f$, and in 
either case with $\delta(\gamma) > \delta$.
\label{pr:conicreg}
\end{Prop}

The two main tools used in proving both of these results are the Mellin transform, and the indicial operator of $A$.
By definition, the Mellin transform of $u(r,z)$ is given by
\[
u(r,z) \longmapsto u_M(\zeta,z) := \int_0^\infty r^{i\zeta-1} u(r,z)\, dr;
\]
this is simply the Fourier transform in logarithmic coordinates $t = - \log r$. The \emph{indicial operator} of $A$ is defined by
\begin{equation}
I(A) : = -(r \del_r)^2 - n r \del_r + \calL_N - (n-1);
\label{eq:indicop}
\end{equation}
this acts on functions on $\RR^+ \times N$ and should be understood as a model for $A$ at $r = 0$. We also write $I(\calL) = r^{-1} I(A) r^{-1}$.  
The fundamental relationship 
connecting these two objects is the formula
\begin{equation}
( I(A)u )_M(\zeta,z) = \left(\calL_N +\zeta^2 + in\zeta - (n-1)\right)u_M(\zeta,z),
\label{eq:indop}
\end{equation}
which is very closely related to \eqref{eq:nantes}. The inverse of the indicial operator provides a good parametrix for $A$. 
This is effective because $I(A)$ is invariant under the $\RR^+$ action, $(r,z) \mapsto (\lambda r, z)$, so its analysis may 
be reduced to that of a family of operators on $N$. Indeed, denote by $I_\zeta(A)$ the operator on the right in \eqref{eq:indop};
this is called the indicial family of $A$. The family of operators $I_\zeta(A)^{-1}$ is essentially the resolvent of $\calL_N$, 
and since $\calL_N$ is a self-adjoint elliptic operator on a compact manifold, it has discrete spectrum, hence this 
family extends meromorphically to all of $\CC$.  Using \eqref{eq:nantes} and 
\eqref{eq:indop}, its poles occur precisely at $\zeta = i\gamma$ where $\gamma$ is an indicial root of $A$.  
A parametrix for $A$ itself may be constructed by a Fourier synthesis of the operators $I_\zeta(A)^{-1}$. 
We give further details on this in the slightly more general setting of depth $2$ iterated edge spaces 
later in this section and refer to \cite{Ma-edge} and \cite{Melrose} for complete proofs.

The important application of this regularity theorem here is to the problem of characterizing self-adjoint extensions of the 
generalized Laplacian $\calL$. Consider the mapping
\[
\calL: L^2(M) \longrightarrow L^2(M).
\]
It is proved in \cite{Gil-Mend} (in the scalar case, but the techniques readily apply to the vector-valued case), see also \cite{Lesch}, that 
if $u \in \calD_{\max}$, i.e.\ $u \in L^2$ and $\calL u \in L^2$,  then using Proposition \ref{pr:conicreg}, 
\begin{equation}
u = \sum_{\gamma_j \in (-n/2, -n/2+2]}  c_j r^{\gamma_j}  \phi_j(z) + w, \qquad w \in r^2 H^2_b(M).
\label{eq:expansion}
\end{equation}
Furthermore, it follows directly from the method of proof that the mappings $\calD_{\max} \ni u \mapsto c_j(u)$ and  $\calD_{\max}\ni 
u \mapsto w \in r^2 H^2_b(M)$ are both continuous. 
In general one would also expect log terms in this expansion, but in our specific applications these appear 
rarely, so for simplicity we omit them in our discussion of the general theory. 

On the other hand, one can also show that $u \in \calD_{\min}$ if and only if all $c_j = 0$, i.e.\ $ u \in r^2 H^2_b(M)$.
Thus it is precisely the indicial roots of $\calL$ in the range $(-n/2, -n/2+2]$, which we call the {\it critical interval},
that prevent $\calL$ from being essentially self-adjoint.  It is clear from the definitions that indicial roots exist in this critical interval 
if and only if $\calL_N$ has eigenvalues lying in the interval $[0,1 - (-n/2+1)^2]$; in particular this can happen only if $n\leq 4$. 

Slightly more generally, various closed extensions are characterized by imposing linear algebraic conditions
on the coefficients $c_j$. For example, if $\calL = \nabla^* \nabla$, then $u \in \calD_{\mathrm{Fr}}$ is
equivalent to $c_j = 0$ for all $j$ such that $\gamma_j \in (-n/2, -n/2+1]$ (so we admit only the indicial
roots in the `upper half' of the critical interval, i.e.\ in $(-n/2+1,-n/2+2]$).  In our specific application below, 
the value $-n/2 + 1$ is an indicial root of multiplicity two, so the corresponding solution has a log term in its
expansion there; the Friedrichs extension allows solutions which have the $r^{-n/2+1}$ term, but not the
$r^{-n/2+1}\log r$, in their expansions. 

The fact that there are only finitely many indicial roots in this interval implies that any choice of domain $\calD$ on
which $\calL$ is self-adjoint has the property that $\calD \subset r^\e H^2_b$ for some $\e > 0$. By the
$L^2$ version of the Arzela-Ascoli theorem, this space is compactly contained in $L^2$, which gives the
\begin{Cor}
Let $(\calL, \calD)$ be a self-adjoint extension of $\calL$ acting on $L^2$. Then this operator has
discrete spectrum consisting only of eigenvalues of finite multiplicity. Each eigenfunction $\phi \in \calD$
is polyhomogeneous.
\end{Cor}

In the proofs of our main results below, we use these expansions (and their analogues for restricted depth $2$ 
iterated edge operators) to show that if $u$ lies in the Friedrichs domain of $\calL$ (where $\calL = P$ or $L$),
then we can justify certain integrations by parts by either showing that certain coefficients $c_j$ vanish, or if they
do not, then by showing that the associated functions $r^{\gamma_j}\phi_j(z)$ have certain special properties. 

\subsection*{Incomplete edge operators}
The next case is when $(M,g)$ has an incomplete edge singularity, i.e.\ the singular stratum $S$ is a closed manifold
and some neighbourhood around it in $M$ is a cone-bundle over $S$ with fibre the truncated cone $C_{0,1}(N)$
over a compact smooth manifold $N$. If $\calL$ is a generalized Laplacian on $(M,g)$ and $\rho$ is the radial 
function on the conic fibres, then we define $A$ as before by $A = \rho \calL \rho$. 

We define the indicial roots of $A$ and $\calL$ exactly as in the conic case. Note that when defining the indicial operator for $A$,
we also drop the $\rho \del_y$ derivatives, where $y$ is a coordinate along $S$, since in terms of formal power series
(or polyhomogeneous) expansions, these are lower order terms.  However, these tangential derivative terms are quite important
in other ways, and in this setting $I(A)$ does not capture all the appropriate features of $A$ near $S$. Accordingly, we introduce 
the {\it normal operator} of $A$ at any point $p \in S$, defined as 
\begin{equation}
N(A) : = -(\rho \del_\rho)^2 - (n-m) \rho \del_\rho + \calL_N + \rho^2 \Delta_y - (n-m-1),  
\label{eq:normalop}
\end{equation}
where $\calL_N$ is the generalized Laplacian with respect to the induced metric on the conic link $N$ at the point $p$.
It turns out that in many natural geometric problems, including the present setting, the induced metric on the link $N$ 
is independent of $p \in S$, hence the indicial roots of $N(A)$ do not depend on $p \in S$. For simplicity
we assume that this is the case in all that follows. 

Although $N(A)$ looks similar to $A$ itself, the difference is that we are regarding $(\rho,y)$ as global linear variables on the 
half-space $\RR^+ \times \RR^m$, $m = \dim S$, and have dropped all `higher order' terms from $A$.  As with the indicial
operator for conic operators, the inverse (or pseudo-inverse) of $N(A)$ is the main ingredient in the construction of a parametrix for $A$. Understanding
the invertibility of $N(A)$ is feasible because, once again, it has many symmetries, namely it is invariant by translations
in the $y$ variables and by dilations $(\rho,y) \mapsto (\lambda \rho, \lambda y)$, so we can ultimately reduce its study
to that of a family of operators on $N \times \RR^+$. 

Unlike the conic setting, it is not true that 
\begin{equation}
A: \rho^\delta H^2_{\mathrm{e}}(M) \longrightarrow \rho^{\delta}L^2(M)
\label{eq:Lwesp}
\end{equation}
is Fredholm except when $\delta$ takes on one of the special indicial values. Indeed, although the mapping \eqref{eq:Lwesp} 
has closed range when $\delta \neq \delta(\gamma)$ for some indicial root $\gamma$, it has infinite dimensional 
kernel or cokernel, except possibly inside a narrow interval of values of $\delta$. To show that \eqref{eq:Lwesp} 
is Fredholm it is necessary to impose the additional hypothesis that $N(A)$ is invertible on $\rho^\delta L^2$. To prove 
that arbitrary solutions $u \in \rho^\delta L^2$ of $\calL u = 0$ are regular (polyhomogeneous) as $\rho \to 0$, we 
must assume that $N(A)$ is injective on $\rho^\delta L^2$. 

To phrase this properly,  recall first that the volume form with respect to which $\calL$ is symmetric is a smooth 
nonvanishing multiple of $\rho^{n-m-1} d\rho dV_S dV_N$, so the critical power of $\rho$ which is at the `$L^2$ 
cutoff' is different than in the conic case. Thus we now set $\delta(\gamma) := \gamma + (\dim N+1)/2 = \gamma + (n-m)/2$. 
\begin{Prop}
Suppose that $(M,g)$ is a manifold with an incomplete edge singularity, and let $\calL$ and $A$ be as above. Then \eqref{eq:Lwesp}
is Fredholm if and only if 
\[
N(A): \rho^\delta H^2_\edge(\RR^+ \times \RR^m \times N) \longrightarrow \rho^\delta L^2(\RR^+ \times \RR^m \times N)
\]
is invertible. (This implies that $\delta \neq \delta(\gamma)$ for any indicial root $\gamma$ of $A$, so it is unnecessary
to add this as an extra hypothesis.) 
\label{pr:frededge}
\end{Prop}

The proof is more involved than in the conic case, and relies on a parametrix $G$ for $A$ which is constructed in
the pseudodifferential edge calculus as developed in \cite{Ma-edge}. The corresponding parametrix in the conic 
case is easier since it is constructed from $I_\zeta(A)^{-1}$, which as we have seen is essentially just the resolvent 
family of $\calL_N$, hence well-understood.  In this edge setting, however, one must first analyze the inverse of $N(A)$ 
and then develop a formalism to localize along the edge $S$ and regard $A$ as a perturbation of 
the family of normal operators $N_q(A)$, $q \in S$. Again in contrast to the conic setting, $N(A)$ is not invariant 
under dilations in $\rho$, so the Mellin transform is not the only tool needed; indeed, we must also exploit the translation
invariance of $N(A)$ in the tangential variable $y$. Taking the Fourier transform in $y$ reduces $N(A)$ to 
\[
\widehat{N}(A) = -(\rho \del_\rho)^2 - (n-m) \rho \del_\rho + \calL_N - (n-m-1 - \rho^2 |\xi|^2),
\]
which is a family of $b$-operators near $\rho=0$, with indicial roots independent of the Fourier transform variable $\xi$.  However,
there is an additional `Bessel structure' of these operators as $\rho \to \infty$ which must be taken into account, and in addition
we must understand the dependence on $\xi$ of the corresponding family of inverses. The inverse $N(A)^{-1}$ is then constructed 
by inverse Fourier transform in $y$. A further and less obvious step is to analyze the pointwise asymptotic behaviour of the Schwartz 
kernel for $N(A)^{-1}$. We refer to the paper \cite{Ma-edge} (where the above proposition appears as a particular case of Theorem 6.1) for complete details of this construction, as well as an explanation of how 
to pass from $N(A)^{-1}$ to a parametrix for $A$, and for the proof of the mapping properties of this parametrix. 

Suppose, however, that the nullspace of $N(A)$ in $\rho^\delta L^2$ is nontrivial. Because of the translation invariance in 
$y$ of this operator, it is easy to see that this nullspace must be infinite dimensional, and from this it follows that 
$A$ itself has an infinite dimensional nullspace on $\rho^\delta L^2$. Every solution of $Au = 0$ has stable regularity
with respect to $\calV_{\mathrm{e}}$, but most of these solutions are not polyhomogeneous. In this situation, a canonical
choice of solution to $Au = f$ which has the maximal regularity is the one which is orthogonal (in $\rho^\delta L^2$)
to the nullspace of $A$.  On the other hand, if $N(A)$ is injective on $\rho^\delta L^2$ and has closed range but is not
surjective, then by similar reasoning one finds that its range has infinite codimension, and that the same is true for $A$ too.
All of this can be proved using similar parametrix constructions. Indeed, if $N(A)$ is either injective with closed range or surjective 
on $\rho^\delta L^2$, but not necessarily invertible, then the same type of construction as above, i.e.\ taking tangential
Fourier transform, writing down the solution operator to the resulting ordinary differential operator, etc., yields 
a left or right inverse, respectively, for $N(A)$, and from this one can construct a left or right parametrix for $A$ itself. 
This can be used to prove the
\begin{Prop}
Suppose that $(M,g)$ is a manifold with an incomplete edge singularity, with $\calL$ and $A$ as above. Suppose that
$N(A)$ is injective on $\rho^\delta L^2$ and $\delta \neq \delta(\gamma)$ for any indicial root $\gamma$ of $A$. If
$f \in \rho^{\delta} L^2 \cap \calA_{\phg}$, then any solution $u\in \rho^\delta L^2$ to $Au = f$ satisfies 
$u \in \rho^\delta L^2 \cap \calA_{\phg}$. (Incidentally, this is localizable in the sense that if $f$ is only polyhomogeneous 
in some open neighbourhood in the edge $S$, then $u$ is polyhomogeneous in that same neighbourhood.)

On the other hand, suppose that $N(A)$ is surjective on $\rho^\delta L^2$ and $\delta \neq \delta(\gamma)$ for any indicial root $\gamma$ of $A$. Then the range of $A$ on $\rho^\delta L^2$
is closed and has finite codimension. If $f \in \rho^\delta L^2 \cap \calA_{\phg}$ lies in this range, then the solution $u \in \rho^\delta L^2$
which is orthogonal to the nullspace of $A$ satisfies $u \in \rho^\delta L^2 \cap \calA_{\phg}$. 
\label{pr:regedge}
\end{Prop}

We now discuss closed and self-adjoint extensions of the unbounded mapping 
\begin{equation*}
\calL:  L^2(M) \longrightarrow L^2(M).
\end{equation*}
Assume that $N(\calL): L^2 \to \rho^{-2}L^2$ is surjective, and let 
$G$ be the right (pseudo-)inverse provided by Proposition \ref{pr:regedge} for $\calL: L^2 \to \rho^{-2}L^2$. If $f \in L^2$ is in its range, then
$u = Gf$ lies in $L^2$ and is orthogonal to the nullspace of $\calL$, hence has the maximal regularity amongst all solutions
of this equation. If, moreover, $f \in L^2 \cap \calA_{\phg}$, then $u$ is also polyhomogeneous, so that
\begin{equation}
u = \sum_{\gamma_j \in (-(n-m)/2, -(n-m)/2+2]}  c_j(y) \rho^{\gamma_j}  \phi_j(z) + w, 
\label{eq:expansion2}
\end{equation}
where $w \in \rho^2 H^2_\mathrm{e}(M) \cap \calA_{\phg}$ and every coefficient $c_j(y)$ lies in $\calC^\infty$.  This expansion
can be used to characterize the minimal extension. Indeed, by taking graph closures we can prove that $u \in \calD_{\min}$ 
if and only if $u \in \rho^2 H^2_{\mathrm{e}}(M)$, i.e.\ $\calD_{\min} = \rho^2 H^2_{\mathrm{e}}$.  Note however, that if we only 
assume that $f \in L^2$, but not that $f$ is polyhomogeneous, then the coefficient functions $c_j(y)$ are no longer smooth and 
this expansion is valid only in a weak sense, i.e.\  it holds only if $u$ is paired with a test function in the $y$ variables.  
Nonetheless, it is still meaningful to specify a closed extension of the core domain of smooth functions with support 
in the regular set of $M$, or equivalently a domain $\calD$ such that $(\calL, \calD)$ is closed, by requiring that some 
of these coefficients, or linear combinations of them, vanish.  In our main application below we use the Friedrichs extension, 
which imposes the condition that $c_j = 0$ for all $j$ with $\gamma_j < -(n-m)/2 + 1$, and (since $-(n-m)/2 + 1$ is a 
double indicial root) which also omit the term $r^{-(n-m)/2 + 1}\log r$.

Suppose now that $(\calL, \calD)$ is a self-adjoint extension. In this edge setting it is no longer necessarily true that $(\calL,\calD)$ gives a Fredholm problem; in the conic case this was guaranteed by the fact that $\calD$ includes compactly
in $L^2$, but here this may not be the case.  The criterion for this to hold is simple enough, fortunately. Since we
are only considering domains $\calD$ which are defined by imposing vanishing conditions on certain of the coefficients
$c_j(y)$ in the expansion, we may also impose these boundary conditions on solutions of $N(\calL) u = 0$.  Then $(\calL, \calD)$ is Fredholm if and only if there exist no tempered solutions of $N(A)u = 0$ with $u$ 
satisfying the boundary conditions corresponding to $\calD$. 

To prove this last statement, it suffices to construct a parametrix $G$ in the pseudodifferential edge calculus which has the 
property that $G: L^2(M) \to \calD$ and such that $\calL G - \mbox{Id}$ is a compact operator. The construction is very similar 
to the one sketched above; the main difference is that one needs to use the inverse for $N(\calL)$ which satisfies the boundary 
conditions imposed by $\calD$. We shall provide more details about this for the specific case of the Friedrichs extension for 
the operator $P$ in the next section. It is necessary to analyze $P$ in this way because as we shall show, for this particular operator,
there is no weight $\delta$ such that $N(P): \rho^\delta H^2_{\mathrm{e}} \to \rho^{\delta-2}L^2$ is invertible: it is injective 
when $\delta >1$, and surjective when $\delta < 1$ (and nonindicial).  Nonetheless, the Friedrichs extension does turn out to be Fredholm, and in fact an isomorphism.

\subsection*{Incomplete iterated edge operators}\label{depth2}
The final case to consider is when $(M,g)$ is an incomplete iterated edge space of restricted depth $2$ 
(which means that the depth $2$ stratum consists only of  isolated points). Hence if $q$ is any point in
this depth $2$ stratum, then $q$ has a neighbourhood identified with the truncated cone $C_{0,1}(N)$ 
where $N$ is a compact space with conic (or more generally, incomplete edge) singularities. In our 
setting $N$ has only conic singularities, so for simplicity we assume that this is the case. The results we need do
not appear elsewhere in the literature yet, so we provide more details of the proofs than we did above
for conic or incomplete edge singularities. 

Thus let $(M,g)$ be as above and $\calL$ a generalized Laplacian on $M$. Rather than discussing Fredholm 
properties for $\calL$ acting between weighted Sobolev spaces with different weights as in \eqref{eq:mapedge},
we focus immediately on the analysis of the unbounded map
\begin{equation}
\calL: L^2(M) \longrightarrow L^2(M).
\label{eq:igny2}
\end{equation}
Our goal is to define certain domains on which $\calL$ is self-adjoint, and in particular, to characterize
elements in the Friedrichs domain.  We shall also prove a regularity theorem, showing that certain solutions 
of $\calL u = f$ with $f \in L^2 \cap \calA_{\phg}$ are also polyhomogeneous. Much of this is analysis is local, 
and since we have already studied this problem near the edges, we consider the behaviour of $\calL$ on 
sections supported near the depth $2$ conic points. 

Thus if $q$ is a depth $2$ point and $N$ its link, then writing $\calL$ as in (\ref{eq:mgL}) (now, of course, 
there are no $y$ variables) exhibits the induced operator $\calL_N$ on $N$. Since $N$ has conic singularities,
$\calL_N$ is a conic operator on a compact space, so we can bring to bear the theory we have developed for 
such operators.  In fact, in order to mimic that theory here, we would like to consider $\calL_N$ as
a self-adjoint operator, and for that we must choose a domain $\calD_N$ such that $(\calL_N,\calD_N)$
is self-adjoint on $L^2(N)$, and hence has discrete spectrum.  For example, we can let $\calD_N$ be
the Friedrichs domain. In any case, having made this choice, we define the indicial operator $I(\calL)$ 
as in \eqref{eq:indicop} ; the indicial roots relative to $\calD_N$ are the values $\gamma$ 
for which there exists $\varphi \in \calD_N$ such that $\calL(r^\gamma \varphi) = \calO(r^{\gamma-1})$.

The next step is to define the function space $L^2_{\calD_N}(M)$; this consists of all functions $u \in L^2$ which,
near any depth $2$ point $q$, have the property that $u(r, \cdot) \in \calD_N$ for almost every $r \in (0,1)$, and in addition
\[
\int_0^1 ||u(r, \cdot)||_{\calD_N}^2 r^{n-1}\, dr < \infty.
\]
The Sobolev spaces $H^\ell_{\ie,\calD_N}$ are then defined as the intersection $H^\ell_\ie(M) \cap L^2_{\calD_N}(M)$.
Using the expansion \eqref{eq:expansion}  for elements of $\calD_N$ for each $r$, and the continuous dependence
of this expansion on the function $u(r,\cdot) \in \calD_N$, we see that 
if $u \in L^2_{\calD_N}$, then
\begin{equation}
u = \sum c_j(r) s^{\gamma_j} \phi_j(\theta) + w, \qquad \mbox{where}\qquad 
w \in L^2((0,1), r^{n-1}\, dr; s^2 H^2_b(N)),
\label{eq:expedges}
\end{equation}
and each $c_j \in L^2(r^{n-1}\, dr)$ as well. Here the $\gamma_j$ are indicial roots lying in the critical interval
for $\calL_N$ along the edge and $\sum c_j(r) s^{\gamma_j}\phi_j(\theta)$ lies in $\calD_N$ for almost every $r$.

We shall suppose for simplicity that $\calL = -\del_r^2 - \frac{n-1}{r}\del_r + \frac{1}{r^2}\calL_N + \frac{1}{r}E$, 
where the remainder term $E$ has the special form
\begin{equation}\label{eq:exprrem}
E = f_1(r)\calL_N + f_2(r) r\del_r + f_3(r);
\end{equation}
the matrix-valued functions $f_i$ are assumed to be smooth up to $r=0$.  
This may seem restrictive, but is satisfied by the operators $P$ and $L$ in our problem, and simplifies various
arguments below. 

We also assume that $\calD_N$ has the property that the coefficients in the expansion (\ref{eq:expedges}) of 
$u \in L^2_{\calD_N}$ along each edge satisfy certain vanishing conditions which correspond to a self-adjoint
Fredholm extension of $\calL$ along these edges, as explained in the previous subsection. Thus the
global definition of $L^2_{\calD_N}$ includes the imposition of these boundary conditions along all the edges.

We first prove an existence theorem.
\begin{Prop}
Let $(M,g)$, $\calL$ and $\calD_N$ satisfy all of the conditions above.  Then there is an operator $G$ such that
\[
G: L^2(M) \longrightarrow r^2 H^2_{\ie,\calD_N}(M)
\]
is bounded, and $\calL G - \mbox{Id}$ is compact on $L^2(M)$. 
\label{pr:fredsv}
\end{Prop}
\begin{proof}
Using the assumption on $\calD_N$ along the edges, we can construct a parametrix for $\calL$ with compact
remainder for sections supported away from the depth $2$ points. Thus it suffices to produce a parametrix
near each such point. First rewrite $\calL u = f$ as $A v = h$ where $v = r^{-1}u$, $h = rf \in r L^2$.
A priori, $v = r^{-1}u \in r^{-1}L^2$, though in fact, by Proposition \ref{pr:nantes2}, $v \in r^{-1}H^2_\ie(M)$. 

Consider $A$ as a perturbation of its indicial operator at $r=0$, $A = I(A) + r E'$, where $E'=r^{-1}Er$ has the same
structure as $E$ in \eqref{eq:exprrem}.  It suffices to construct an inverse for the first term $I(A)$.  
To solve $I(A) v = h$, take the Mellin transform:
\[
I_\zeta(A)\, v_M = h_M.
\]
Since $\calL_N : \calD_N \to L^2(N)$ is self-adjoint with discrete spectrum, the indicial operator  $I_\zeta(A) = 
\calL_N +\zeta^2 + in\zeta - (n-1) : \calD_N \to L^2(N)$ is invertible for every $\zeta$ outside of a discrete set.
The inverse $I_\zeta(A)^{-1}$, which is (up to a change of variables in $\zeta$) the resolvent of $\calL_N$, is a family 
of $b$-pseudodifferential operators which is meromorphic in $\zeta$ with poles at the points $\zeta = i\gamma$ 
where $\gamma$ is an indicial root of $A$; furthermore, for each $\zeta$ in the
regular set, $I_\zeta(A)^{-1}$ maps $L^2(N)$ to $\calD_N$.  

All but a discrete set of horizontal lines $\mbox{Im}\, \zeta = \mu$ are pole-free.  Along any line $\mathrm{Im}\,\zeta = \mu$
away from these exceptional heights,  we have the two estimates:
\begin{equation}
|| I_\zeta(A)^{-1} ||_{L^2(N) \to L^2(N)}  \leq C (1+|\zeta|)^{-2}, \qquad 
|| I_\zeta(A)^{-1}||_{L^2(N) \to \calD_N}  \leq C,
\label{resests}
\end{equation}
where $C$ is bounded so long as the line of integration remains bounded away from any horizontal line containing a pole.

Along any pole-free line, define
\[
\calG_\mu :=  \int_{\mbox{Im}\, \zeta = \mu} r^{-i\zeta} I_{\zeta}(A)^{-1}\, d\zeta.
\]
The Plancherel formula for the Mellin transform, the fact that this transform exchanges $r\del r$ with $-i\zeta$,
and the two estimates from \eqref{resests}, show directly that
\[
\calG_\mu: L^2(\RR^+, r^{-1-2\mu}dr; L^2(N))  \longrightarrow L^2(\RR^+, r^{-1-2\mu}dr ; \calD_N) \cap H^2_b(\RR^+, 
r^{-1-2\mu}dr ; L^2(N))
\]
is a bounded mapping. The space $H^2_b$ on the right consists of functions $u$ for which $u, r\del_r u$ and $ (r\del_r)^2 u$ 
all lie in $L^2$ with respect to the given measure.  Note also that by the meromorphy of $I_\zeta(A)^{-1}$, $\calG_\mu$ is locally 
constant in $\mu$, so long as we stay away from lines containing poles. We thus obtain a family of inverses for $I(A)$ acting on 
various weighted function spaces. In particular, $I(A) \circ \calG_{-n/2}$ is the identity on $L^2(\RR^+,r^{n-1}dr; L^2(N))$. 

To conclude, first observe that $A \circ \calG_{-n/2} = I - K$, where $K = - r E' \calG_\mu$. By \eqref{eq:exprrem}, this error term $K$ is clearly bounded on $L^2(\RR^+, r^{n-1}dr; L^2(N))$, and if we restrict to a region $r \leq r_0 << 1$,
$K$ has small norm on this space. Choosing $r_0$ small, we can invert $I - K$ by a Neumann series and obtain an
exact right inverse $G'$ for $A$ on $L^2([0,r_0], r^{n-1}dr; L^2(N))$, and hence corresponding right inverse $G^{(1)}=rG'r$ for $\calL$.

Now choose functions $\chi_1(r)$ and $\tilde{\chi}_1(r)$, where $\chi_1(r)$ equals $1$ in $r \leq r_0/2$ and vanishes for 
$r \geq 3r_0/4$ while $\tilde{\chi}_1 = 1$ on the support of $\chi_1$ and vanishes when $r \geq r_0$; define $\chi_2(r) = 1 - \chi_1(r)$ 
and choose $\tilde{\chi}_2$ which equals $1$ on the support of $\chi_2$ and vanishes in $r \leq r_0/4$. Finally, if $G^{(2)}$ is
a parametrix for $\calL$ with compact remainder (we could even choose it to be an exact inverse) on the region $r \geq r_0/2$,
then we set 
\[
G = \tilde{\chi}_1 G^{(1)} \chi_1 + \tilde{\chi}_2 G^{(2)} \chi_2.
\]
A standard calculation and argument, using the boundedness and regularizing properties of the $G^{(j)}$, shows that
$\calL G = I - \mathcal{K}$, where $\mathcal{K}$ is compact on $L^2(M)$. 
\end{proof}
This proof is precisely the same as in the conic case, but we have had to impose conditions on $\calD_N$ to ensure
good control of this operator near the edges. 

We now turn to regularity of solutions. 
\begin{Prop}
Assume all the same conditions on $(M,g)$, $\calL$ and $\calD_N$. Then any solution $u \in L^2_{\calD_N}(M)$ of $\calL u = f$ 
with $f \in L^2(M)$ has an expansion as $r \to 0$ of the form
\[
u = \sum_{\gamma_j \in (-n/2, -n/2 + 2]} c_j r^{\gamma_j} \phi_j(z) + w, \qquad w \in r^2 H^2_{\ie,\calD_N}(M), 
\]
where each $\phi_j(z)$ is an eigenfunction for $(\calL_N,\calD_N)$ associated to the eigenvalue corresponding to $\gamma_j$. 

Finally, if $f \in L^2(M) \cap \calA_{\phg}$ and $u$ is polyhomogeneous near the singular edges, then $u$ is also polyhomogeneous 
near the singular vertices, i.e. $u \in L^2 \cap \calA_{\phg}$.
\label{pr:polysv}
\end{Prop}
\begin{proof}
Note first that we only need to focus on a neighbourhood of $r=0$ (i.e. the depth $2$ points), since the first assertion
only concerns these regions anyway and for the second assertion the assumption on $\calD_N$ already 
guarantees that $u$ is polyhomogeneous near all edges. 

As above, write $\calL u = f$ as the equivalent equation $A v = h$, so $I(A) v = h - r E' v$; by the remarks
above, and the assumption \eqref{eq:exprrem} on $E$, the right hand side lies in $L^2$. Passing
to the Mellin transform, $(h-rEv)_M(\zeta, \cdot) \in L^2(N)$ and hence 
\begin{equation}
v_M(\zeta,\cdot) = I_\zeta(A)^{-1}(h-rEv)_M(\zeta,\cdot)
\label{eq:Mel}
\end{equation}
makes sense and takes values in $\calD_N$ for each $\zeta$. 

A priori, $v_M(\zeta,z)$ is holomorphic in the lower half-plane $\Im \zeta < -n/2 - 1$ (since $v \in r^{-1} L^2$)
with values in $\calD_N$. On the other hand, $(h - r Ev)_M$ is holomorphic in $\Im \zeta < -n/2$ with values in $L^2(N)$. 
Furthermore, $I_\zeta (A)^{-1}$ is meromorphic in all of $\CC$, with poles at the indicial roots of $A$, and takes values in
the space of bounded operators from $L^2(N)$ to $\calD_N$.  This proves that $v_M(\zeta,z)$ extends meromorphically 
to $\Im \zeta < -n/2$, with at most a finite number of poles in the strip $\Im \zeta \in [-n/2 - 1, -n/2)$, and with values
in $\calD_N$. The residues at these poles lie in the eigenspaces of $\calL_N$ corresponding to those indicial roots, hence 
are polyhomogeneous on $N$. Taking the inverse Mellin transform shows that $v = \sum r^{\gamma_j}\phi_j(z) + w$ where 
$w \in L^2(r^{n-1}dr; \calD_N)$. One further iteration of this argument shows that $v$ has the decomposition asserted 
in the statement of this Proposition.

The next step is to show that if $f$ (and hence $h$) is conormal, then so is $v$ and hence $u$ as well. Recall from
Proposition \ref{pr:nantes2} that $u \in H^\ell_{\ie}$ for any $\ell \geq 0$. Conormality in this setting is the statement 
that $u$ has stable regularity with respect to $r \del_r, s\del_s $ and $\del_\theta$, and the only one of these vector
fields which does not lie in $\calV_{\ie}$ is $r\del_r$.  Therefore we need only consider the behaviour of $u$ under
repeated action of $r \del_r$. To study this, note that since the Mellin transform intertwines $r\del_r$ with 
multiplication by $-i\zeta$, then $(r \del_r)^\ell w \in L^2$ for all $\ell \geq 0$ if and only if $|w_M(\zeta,\cdot)| 
\leq C_\ell (1 + |\Re \zeta|)^{-\ell}$ for all $\ell \geq 0$. Therefore, multiplying (\ref{eq:Mel}) by $\zeta^2$ and noting
that $\zeta^2 I_\zeta(A)^{-1}$ is bounded as $|\Re \zeta| \to \infty$, we obtain $(r \del_r)^2 v \in r^{-1}L^2$. 
Now iterate this step, eventually showing that $(r \del_r)^\ell v \in L^2$ for any $\ell \geq 0$. This proves 
conormality.

The final step is to show that $v$ (and hence u) is actually polyhomogeneous if $f$ is. For this
we return to the first argument. Write $v_M$ again as in (\ref{eq:Mel}) and use that $(r Ev)_M$
now extends to be holomorphic and rapidly decaying as $|\Re \zeta| \to \infty$ in $\Im \zeta < -n/2$,
while $h_M$ and $I_\zeta(A)^{-1}$ both extend meromorphically. This gives a meromorphic extension
of $v_M$ to $\Im \zeta < -n/2$ with rapid decay in horizontal directions. Iterating this eventually yields
that $v_M$ is meromorphic in the entire complex plane and rapidly decaying along horizontal lines.
The final observation is that this is precisely the characterization of Mellin transforms of polyhomogeneous
functions. Indeed, one can recover $v$ by taking the inverse Mellin transform initially along any line
$\Im \zeta = c < -n/2 - 1$, then recovering the full expansion of $v$ by moving the contour up; the
residues at the poles correspond to the terms $r^{\gamma} \phi(z)$ in the expansion.
\end{proof}

\begin{Rem}
This type of proof was developed and used as a crucial part of the arguments in \cite{Ma-edge}, but 
this use of Mellin transform and its relationship with polyhomogeneity goes back further to \cite{Mel-Men}.
\end{Rem}

\section{Proof of Theorem \ref{thm:gauge}}\label{sec:proof}
We now turn to the proof of the first main theorem, which implies that if $\dot{g}$ is an infinitesimal deformation 
in standard form with no component corresponding to change of dihedral angles, then there is a unique equivalent 
infinitesimal deformation $h$ which is in Bianchi gauge and such that $h$ is polyhomogeneous. Recall from \eqref{eq:tmeq}
at the end of section \ref{sec:gauge} that $h = \dot{g} - 2 \delta^* \eta$ where $P \eta = B \dot{g}$. Thus we must study 
the equation $P\eta = f$, which we do using the methods and results of the previous section. This requires various 
computations of the indicial roots of $P$ on cone-surfaces and conifolds and of its indicial and normal operator. 
We carry these out first, then complete the proof of Theorem \ref{thm:gauge} at the end of this section.

Let us denote by $\Delta_j$ the Hodge Laplacian on $j$-forms. Because of the relationship of $P = \na \n - \Ric$ with
$\Delta_1 = \na \n + \Ric$,  a certain amount of the analysis below is most neatly phrased in terms of this Laplacian.

\subsubsection*{Analysis of the Hodge Laplacian on cone-surfaces}
The behaviour of $\Delta_0$ is not hard to understand. The following result is classical, but see \cite{Gil-Mend} for 
extensive discussion.
\begin{Prop}\label{pr:Laph0}
Let $(N,h)$ be a compact (possibly with non-constant curvature) cone-surface with all cone angles less than $2\pi$. Then the
corresponding Laplacian $\Delta_0$ has a one-parameter family of self-adjoint extensions, described in the proof 
below. The Friedrichs extension, with domain $\calD_{\Fr}^0$, is the only one which admits the constant function 
$u \equiv 1$ in its nullspace.
\end{Prop}

\begin{proof}
Near any cone point $q$, using the radial and angular variables $(r,\theta)$,  $I(\Delta) = r^{-2}\left(-(r\del_r)^2 - \del_\theta^2\right)$, so the indicial roots are all of the form $n\gamma$, $n \in \Z$ (with $0$ a double root), where $\gamma  = 2\pi /\alpha$.
 The function $r^\mu$ lies in $L^2(r dr d\theta)$ if and only if $\mu > -1$, so the only indicial roots affecting self-adjointness 
are those in the interval $(-1,1]$. Since $\gamma > 1$, this excludes every indicial root except the double root $0$. 
This root has corresponding solutions $r^0 = 1$ and $\log r$. Each of the corresponding domains of self-adjointness is 
then given by the condition $\calD(\omega) = \{u = a + b\log r + v, v \in r^2 H^2_b\}$ where $(a+ib)/|a+ib| = e^{i\omega}$.
 In particular, $\calD_{0,\Fr} = \calD(0)$, i.e.\ corresponds to the case where no log term is allowed.
\end{proof}

Because $\dim N = 2$, the Hodge star intertwines the quadratic forms associated to $\Delta_0$ and $\Delta_2$, so this result 
characterizes the Friedrichs extension of $\Delta_2$ as well.

The story for $\Delta_1$ is slightly more complicated.  Rather than characterizing all possible self-adjoint extensions, we 
focus on the two with particular geometric interest: the Friedrichs extension, with domain $\calD_{1,\Fr} = 
\{u \in L^2(N): \  \n u, \na \n u \in L^2\}$, and the Dirichlet-Neumann extension, with domain $\calD_{1,\DN} = 
\{u \in L^2 : \ du,\ \delta u,\ d \delta u,\ \delta d u \in L^2\}$ (recall from Proposition \ref{thm:stokes} that $\n_{\min} = \n_{\max}$, $d_{\min} = d_{\max}$, and $\delta_{\min} = \delta_{\max}$).  The latter was introduced by Gaffney \cite{Gaffney}. 
The naturality of the Friedrichs extension is obvious, but the Dirichlet-Neumann extension is also important when considering forms of different degrees, since it is the 
one that is intertwined by $d$. Fortunately, in the 
present setting, these extensions agree.
\begin{Prop}
Let $(N,h)$ be a closed (possibly variable-curvature) cone-surface with cone angles less than $2\pi$. Then $\calD_{1,\Fr}
=\calD_{1,DN}$, and we denote this domain simply by $\calD_1$. Furthermore, every $u \in \calD_{1}$ also 
satisfies $\n d u \in L^2$.
\label{pr:Laph1}
\end{Prop}
\begin{proof} Using the local trivialization of the $1$-form bundle given by $dr$ and $rd\theta$, the indicial operator of 
$\Delta_1 = \na \n + \Ric$ is
\begin{equation*}
I(\Delta_1) =  r^{-2}\left(- (r\d_r)^2 - \d_\theta^2 +1\right)\, \begin{pmatrix}1 & 0 \\ 0 & 1\end{pmatrix}  + 
2r^{-2}\d_\theta\,  \begin{pmatrix}0&1\\ -1& 0\end{pmatrix}.
\end{equation*}
The set of indicial roots is thus $\Lambda(\Delta_1) = \{n\gamma \pm 1\ ;\ n\in \Z\}$; the corresponding solutions are of the form
\begin{equation}
\begin{split}
\eta_n^{++} =  r^{n\gamma+1} e^{in\gamma\theta} (dr-ird\theta),  \qquad \eta_n^{+-} = r^{-n\gamma-1} 
e^{in\gamma\theta} (dr-ird\theta), \\
\eta_n^{-+} =  r^{n\gamma-1} e^{in\gamma\theta} (dr+ird\theta), \qquad 
\eta_n^{--} = r^{-n\gamma+1} e^{in\gamma\theta} (dr+ird\theta). 
\end{split}
\label{eq:solindcs}
\end{equation}

Just as for functions, we must examine indicial roots lying in the interval $(-1,1]$. The only possible ones are $1-\gamma$ 
and $\gamma-1$ (and these lie in this interval only if the cone angle $\alpha$ is between $\pi$ and $2\pi$). Now observe that 
\[
d \eta^{++}_{-1} = d\overline{\eta^{--}_{+1}}  =  r^{-\gamma} e^{-i\gamma \theta} ( -i (2-\gamma) -i \gamma)  dr \wedge r d\theta \notin L^2, \qquad 
\mbox{but}\]
\[d\eta^{-+}_{+1} = d\overline{\eta^{+-}_{-1}}  =  r^{\gamma - 2 } e^{i\gamma \theta} (i \gamma  -i \gamma)  dr \wedge r d\theta = 0.
\]

If $\eta \in L^2$ and $\Delta_1 \eta \in L^2$, then 
\[
\eta = c^{++} \eta^{++}_{-1} + c^{--} \eta^{--}_{+1} + c^{-+} \eta^{-+}_{+1} + c^{+-} \eta^{+-}_{-1} + \tilde{\eta}, \qquad \tilde{\eta} \in r^2 H_b^2.
\]
Recall that $\eta \in \calD_{1,\Fr}$ if and only if $c^{++}=c^{--} = 0$, for this is what is needed to ensure that $\n \eta \in L^2$. 
Thus to prove that $\calD_{1,\Fr} = \calD_{1,\DN}$, we must show $c^{++}=c^{--} = 0$ whenever $\eta \in \calD_{1,\DN}$. 
But this follows immediately from the computation above since if $c^{++}$ or $c^{--} \neq 0$, then $d\eta \notin L^2$. 

To prove the last assertion, simply note that if $\eta \in \calD_1$, so $\eta = c^{-+} \eta^{-+}_{+1} + c^{+-}\eta^{+-}_{-1} + \tilde{\eta}$, then 
as we have already checked, $d\eta = d\tilde{\eta} \in r H^1_b$, so $\nabla d \eta \in L^2$, as claimed. 
\end{proof}

\begin{Prop} \label{prop:bashilb}
Let $(N,h)$ be a cone-surface with all cone angles less than $2\pi$.  Using the Friedrichs extension of the Hodge
Laplacian on forms of every degree, let $\{\psi_j\}$ be the orthonormal basis of $L^2(N)$ consisting of eigenfunctions 
of $\Delta_0$, with corresponding eigenvalues $\lambda_j$, so that $\{*\psi_j\}$ is an orthonormal basis of 
$L^2(N, \Lambda^2 T^*N)$ consisting of eigenforms of $\Delta_2$.  For each $j$ with $\lambda_j > 0$, define the $1$-forms
\[
\phi'_j = \frac{1}{\sqrt{\lambda_j}} d\psi_j, \qquad \phi_j'' = \frac{1}{\sqrt{\lambda_j}} *d\psi_j.
\]
Then $\{\phi_j', \phi_j''\}$ is an orthonormal basis of $L^2(N,T^*N)$ consisting of eigenforms for
$\Delta_1$. These satisfy $\delta \phi_j' = \sqrt{\lambda_j} \psi_j$, $d \phi_j'' = \sqrt{\lambda_j} *\psi_j$ and furthermore, 
\[
\n d \psi_j, \n d \phi_j', \n d \phi_j'' \in L^2.
\]
Finally, if the curvature $\K$ of $(N,h)$ is equal to $+1$, then $\mbox{spec}(\Delta_1) \subset (1,\infty)$. 
\end{Prop}

\begin{proof}
This is well-known when $N$ is a smooth closed surface, and the formal algebra is exactly as in that case.
The only point to check is that each form is in the correct operator domain.  Thus fix $\psi \in \calD_0$ with 
$\Delta_0 \psi = \lambda \psi$ and set $\phi = \lambda^{-1/2}d\psi$.  By definition, $\phi \in L^2$, and 
$d\phi = 0$, so $\delta d \phi = 0$ too, hence both $d\phi, \delta d \phi  \in L^2$.  Next, $\delta \phi = 
\lambda^{-1/2}\Delta_0 \psi = \lambda^{1/2} \psi \in L^2$, and finally $d \delta \phi = \lambda \phi \in L^2$.  
This proves that $\phi \in \calD_1$, as claimed.  

The previous proposition now gives that $\n d \psi \in L^2$. The other related assertions follow in the same way. 

The lowest eigenvalue $\lambda_1$ for $\Delta_1$ is strictly positive. Using the Weitzenb\"ock formula $\Delta = \na \n + \Ric$,
the fact that we can integrate by parts for forms in the Friedrichs domain, and that $\Ric = 1$ here, we see that if
$\lambda_1 \phi_1 = \na \n \phi_1 + \phi_1$, then $(\lambda_1 - 1) ||\phi_1||^2 = ||\n \phi_1 ||^2$, hence $\lambda_1 \geq 1$.  
If $\lambda_1 = 1$ then $\n \phi_1 \equiv 0$, but $\phi_1$ decays at each conic point, which is a contradiction; thus $\lambda_1 > 1$. 
\end{proof}

Although we will only use this estimate on the eigenvalues, it is worth mentioning that a sharper one is true.
This is due to Hartmut Wei\ss{} and the proof appears in \cite{MW}.
\begin{Prop}[Wei\ss{}]
If $(N,h)$ is a spherical cone-surface with all cone angles less than $2\pi$, then the smallest nonzero eigenvalue for the 
Friedrichs extension of the scalar Laplacian on $N$ satisfies $\lambda_1 \geq 2$, with equality if and only if $(N,h)$ is 
isometric to the spherical suspension of the circle $\Sph^1_\alpha$.
\label{pr:weiss}
\end{Prop}

\subsection*{Computations near a singular edge}
We now turn to computations for the operator $P$ on the three-dimensional cone-manifold $M$ near a singular edge, 
with dihedral angle $\alpha \in (0,2\pi)$. Recall \eqref{eq:metric1conemfd} that the local expression of the metric 
in cylindrical coordinates is $d\rho^2 + \sn_{\K}^2 \rho \, d\theta^2 + \cs_{\K}^2 \rho \, dy^2$.

Any $1$-form $\eta$ can be written in a neighbourhood of a point on this edge as $f d\rho + g \rho d\theta + h dy$, so that 
the pointwise norm of $\eta$ is comparable to $(|f|^2 + |g|^2 + |h|^2)^{1/2}$. In this basis, both the Hodge Laplacian on $1$-forms,
$\Delta_1$, and $P$ are incomplete elliptic edge operators. The indicial and normal operators for either of these are the same as 
for the corresponding operators in the flat case, so are all the same as for $\na \n$ on the flat edge $C(\Sph^1_\alpha) \times \RR$. 

We first calculate that
\begin{equation*}
I(\na\n)(f d\rho + g \rho d\theta + h dy) = I( (\na \n)^{C(\Sph^1_\alpha)}) (f d\rho + g \rho d\theta) + (I((\Delta_0)^{C(\Sph^1_\alpha)})\, h)\, dy.
\end{equation*}
Thus using the computations of the last subsection, we see that the set of indicial roots is given by
\begin{equation}\label{eq:rootedge}
\Lambda_{\mathrm{e}}(\na \n) = \Lambda_{\mathrm{e}}(P) = \{n\gamma +1 , n\gamma -1, n\gamma; \ n\in \Z \},
\end{equation}
with corresponding solutions the $1$-forms $\eta_n^{\pm\pm}$ from \eqref{eq:solindcs} in the first two components and
$r^{\pm n\gamma} e^{in\gamma \theta} dy$ in the last component. For $n = 0$, the solutions are $dy$ and $\log r\, dy$. 

The only indicial roots lying in the critical interval $(-1,1]$ are $0$ (with multiplicity $2$) and also $1-\gamma$, $\gamma-1$
if the cone angle is between $\pi$ and $2\pi$. 

To proceed further, we must analyze the normal operator. 
\begin{Prop}
Fix $\delta > 1$ with $\delta -1 \notin \Lambda_{\mathrm{e}}(\na \n)$. Then 
\[
N(\na \n): \rho^\delta H^2_{\mathrm{e}}(C(\Sph^1_\alpha) \times \RR) \longrightarrow \rho^{\delta-2}L^2(C(\Sph^1_\alpha) \times \RR)
\]
is injective with closed range; moreover, for the same values of $\delta$, 
\[
N(\na \n): \rho^{2-\delta} H^2_{\mathrm{e}}(C(\Sph^1_\alpha) \times \RR) \longrightarrow \rho^{-\delta}L^2(C(\Sph^1_\alpha) \times \RR)
\]
is surjective.
\label{pr:nopis}
\end{Prop}
\begin{proof}
We begin by calculating that
\begin{equation*}
N(\na\n)(f d\rho + g \rho d\theta + h dy) = 
N( (\na \n)^{C(\Sph^1_\alpha)\times \R}) (f d\rho + g \rho d\theta) + \left(N((\Delta_0)^{C(\Sph^1_\alpha)\times \R}) h \right) \, dy
\end{equation*}
and computing the homogeneous solutions of $N(\na \n)\eta = 0$. 

Suppose that $\eta \in \rho^\delta L^2$ is a solution of $N(\na\n) \eta = 0$. Decomposing into Fourier series in $\theta$ and 
taking Fourier transform in $y$ gives the family $\hat \eta_n(\rho,\xi) = \hat f_n d\rho + \hat g_n\, \rho d\theta + \hat h_n \,dy$
which satisfies the partially coupled equations
\begin{equation*}
\begin{cases}
\left(-(\rho\d_\rho)^2 +n^2\gamma^2+1+\rho^2\xi^2\right)\hat f_n +2in\gamma\hat g_n = 0\\
\left(-(\rho\d_\rho)^2 +n^2\gamma^2+1+\rho^2\xi^2\right)\hat g_n -2in\gamma\hat f_n = 0
\end{cases}
\end{equation*}
and
\begin{equation*}
\left(-(\rho\d_\rho)^2 +n^2\gamma^2+\rho^2\xi^2\right)\hat h_n = 0.
\end{equation*}

These can be solved explicitly.  Using complex conjugation, assume that $n \geq 0$.  When $\xi = 0$, these equations
reduce to the indicial equations above, and we have already obtained the solutions; none are in $\rho^\delta L^2$ globally
in $\rho$. If $\xi \neq 0$ and $n \neq 0$, then
\begin{align*}
\hat f_n &= c_n^1(\xi) I_{n\gamma-1}(\rho |\xi|) + c_n^2(\xi) K_{n\gamma-1}(\rho |\xi|) + 
c_n^3(\xi) I_{n\gamma+1}(\rho |\xi|) + c_n^4(\xi) K_{n\gamma+1}(\rho |\xi|) \nonumber\\
\hat g_n &= i c_n^1(\xi) I_{n\gamma-1}(\rho |\xi|) + ic_n^2(\xi) K_{n\gamma-1}(\rho |\xi|) - 
ic_n^3(\xi) I_{n\gamma+1}(\rho |\xi|) - ic_n^4(\xi) K_{n\gamma+1}(\rho |\xi|) \nonumber\\
\hat h_n &= c_n^5(\xi) I_{n\gamma}(\rho |\xi|) + c_n^6(\xi) K_{n\gamma}(\rho |\xi|),
\end{align*}
while if $n=0$, then
\begin{align*}
\hat f_0 &= c_0^1(\xi) I_{1}(\rho |\xi|) + c_0^2(\xi) K_{1}(\rho |\xi|) \nonumber\\
\hat g_0 &= c_0^3(\xi) I_{1}(\rho |\xi|) + c_0^4(\xi) K_{1}(\rho |\xi|) \nonumber\\
\hat h_0 &= c_0^5(\xi) I_{0}(\rho |\xi|) + c_0^6(\xi) K_{0}(\rho |\xi|),
\end{align*}
where $c_n^1,\ldots,c_n^6$ are arbitrary functions or distributions and $I_a$, $K_a$ are the modified Bessel functions of 
first and second kind (also called McDonald functions or just Bessel functions of imaginary argument). Their
asymptotics, 
\[
I_a(x) \sim \frac{1}{2^a \Gamma(a+1)}\,x^a, \qquad 
K_a(x) \sim \begin{cases}\Gamma(a)2^{a-1}\,x^{-|a|} &\mbox{ if  } a\neq 0 \\
-\log (x/2) &\mbox{ if }a=0,\end{cases}, \qquad x \to 0,
\]
\[
I_a(x) \sim \frac{1}{\sqrt{2\pi x}}\,e^x, \qquad K_a(x) \sim \sqrt{\frac{\pi}{2x}}\,e^{-x}, \qquad x \to \infty
\]
can be found in \cite{Le}. 
 
Since $\eta \in \rho^\delta L^2$, we can immediately rule out the exponentially growing solutions involving $I_a$, so 
$c_n^1, c_n^3, c_n^5 \equiv 0$. Thus $\hat \eta_n$ decays exponentially as $\rho \to \infty$, while as $\rho \to 0$,
if $n \neq 0$, then $\hat \eta_n \sim \rho^\mu$ ,  where $\mu = -|n|\gamma -1$, $-|n|\gamma$ or $-|n|\gamma+1$,
whereas $\hat \eta_0 \sim \rho^{-1}$ or $\log \rho$. If $1-\gamma < \delta < 1$, then $\hat \eta_0 =
K_0(\rho |\xi|)\, dy \in \rho^\delta L^2(\rho d\rho d\xi)$; however, if $\delta \geq 1$, then none of the $\hat \eta_n$
lie in this space.  This proves that $N(\na \n)$ is injective on $\rho^\delta L^2$ whenever $\delta > 1$.  

The fact that $N(\na \n)$ acting on $\rho^\delta L^2$ has closed range follows directly from the general argument
in \cite{Ma-edge}, see in particular Theorem 5.16 of that paper. That argument involves the construction of a left 
inverse for $N(\na \n)$ which is bounded on $L^2$. A crucial part of this is that $\delta-1$ is not an indicial root. 
For the reader's convenience, let us sketch this argument. Introducing Fourier series in $\theta$ and Fourier transform 
in $y$, we decompose the problem into a family of ordinary differential operators; it is sufficient to construct a left inverse 
for each of these and then show that the norms of these left inverses are bounded independently of $n$ and $\xi$.  
Each of the ODE's which arise in this decomposition are Bessel operators, and one can write these left inverses quite explicitly,
as in \eqref{Green} below. A direct inspection of these Schwartz kernels leads to the uniformity in $n$ and $\xi$ of
their operator norms. One can proceed in exactly the same way to construct a right inverse for $N(\na \n)$ acting on 
$\rho^{2-\delta}L^2$. This shows that the two mappings in the statement of this Proposition both have closed
range and are, respectively, injective and surjective. 
\end{proof}

The two assertions, about injectivity and surjectivity, in this Proposition are equivalent to one another by 
a duality argument. Indeed, using the natural pairing $\rho^{\delta'} L^2 \times \rho^{-\delta'}L^2 \to \CC$ for any 
$\delta'$, if $u \in \rho^{\delta}H^2_{\mathrm{e}}$ and $v \in \rho^{2-\delta}H^2_{\mathrm{e}}$, then 
$\langle N(\na \n)u, v \rangle = \langle u, N(\na \n) v \rangle$. The pairing on the left is between $\rho^{\delta-2}L^2$ 
and  $\rho^{2-\delta}L^2$ while the one on the right is between $\rho^\delta L^2$ and $\rho^{-\delta} L^2$.  
This shows that the dual of $N(\na \n)$ acting on $\rho^\delta L^2$ is naturally identified with the same
operator acting on $\rho^{2-\delta}L^2$, and hence the first is injective with closed range if and only if
the second one is surjective.  As we have remarked earlier, the first mapping has an infinite 
dimensional cokernel, and equivalently, the second mapping has an infinite dimensional nullspace, hence this 
particular problem does not satisfy the hypotheses described in \S 5 that we can choose a value of the weight 
parameter so that $N(\na \n)$ is an isomorphism. Thus Proposition \ref{pr:frededge} does not directly apply, but
as indicated earlier, we can construct a parametrix relative to a different choice of boundary condition.

\subsection*{Modified parametrix construction}
We now wish to construct a parametrix for the Friedrichs extension of $P$. 

The first point is that the difficulty with the choice of weight parameters is absent if we work on the subspace 
$(\rho^\delta L^2)^\perp$ consisting of $1$-forms which are in $\rho^\delta L^2$ and for which the coefficient 
of $e^{in\gamma \theta}$ vanishes when $n=0$.  The construction of an inverse using Fourier analysis and ODE
methods indicated above shows that $N(\na \n): (\rho^\delta H^2_{\mathrm{e}})^\perp \to (\rho^{\delta-2} L^2)^\perp$ 
is an isomorphism when $\delta \in (2-\gamma, \gamma)$. Therefore it suffices to consider the restriction 
of the normal operator to the complementary subspace $(\rho^\delta L^2)_0$ consisting of forms for which only
the $n=0$ Fourier coefficient is nonzero. On this subspace, the only difficulty occurs in the third component of
the $1$-form, since the operator on that piece is the normal operator for the scalar Laplacian,
$N(\Delta_0) = -\rho^{-2}\left((\rho \del_\rho)^2 + \rho^2 \del_y^2\right)$.

Let $f \in L^2(\RR^+ \times \RR, r dr dy)$.  We shall construct a solution 
$u = G_0 f$ to $N(\Delta_0)u = f$ with $u \in \rho^{1-\e}L^2$ near $\rho=0$ for any $\e > 0$. Because $N(\Delta_0)$
is surjective and has infinite dimensional kernel on  this particular weighted space, there are many different solutions 
to this problem, each of which has an asymptotic expansion of the form $u \sim a(y) + b(y) \log \rho + \tilde{u}$ where 
$\tilde{u} \in \rho^{\delta}L^2$. The solution $u$ which we seek is characterized by the condition that $b(y) \equiv 0$.  
The uniqueness of this solution is immediate since we already checked that the only solutions of $N(\Delta_0)u = 0$ 
in $\rho^{1-\e}L^2$ correspond to $K_0(\rho |\xi|)$, and this has a log term in its expansion. 

The construction of $G_0$ proceeds by first taking the Fourier transform in the $y$ variable, and then 
expressing a right inverse for the resulting operator $\widehat{N}(\Delta_0) = \rho^{-2}\left((-\rho \del_\rho)^2 + \rho^2 \xi^2\right)$
in terms of the homogeneous solutions for this operator, $I_0(\rho |\xi|)$, $K_0(\rho |\xi|)$.  The tempered solution
to $\widehat{N}(\Delta_0)\hat{u} = \hat{f}$ which has no $\log \rho$ term is given by integrating in $\rho'$ (with respect to $\rho'd\rho'$) against the Green function
\begin{equation}
\widehat{G}_0(\rho,\rho',\xi) := - I_0(\rho |\xi|) K_0(\rho'|\xi|) H(\rho' - \rho) - K_0(\rho |\xi|) I_0(\rho'|\xi|) H(\rho - \rho')
\label{Green}
\end{equation}
where $H$ denotes the Heaviside function, i.e. 
$$\widehat{G}_0 \hat{f} = -I_0(\rho |\xi|) \int_\rho^\infty K_0(\rho'|\xi|)\hat{f}(\rho',\xi)\rho'd\rho'  - K_0(\rho |\xi|)  \int_0^\rho I_0(\rho'|\xi|)\hat{f}(\rho',\xi)\rho'd\rho',$$
and from this we obtain that the Schwartz kernel of $G_0$ itself is given by
\[
G_0(\rho,\rho',y,y') = \frac{1}{2\pi} \int_{-\infty}^\infty e^{i(y-y')\xi} \widehat{G}_0(\rho, \rho', \xi)\, d\xi.
\]

If $f$ has compact support in the $y$ direction, so that $\hat{f}(\rho,\xi) \in L^2$ decays rapidly as $|\xi| \to \infty$, then using the above expression and the asymptotics of $I_0$ and $K_0$, we see that 
$\hat{u}_0 = \widehat{G}_0 \hat{f} \in \rho^{1-\e}L^2$ for any $\e > 0$, and $\hat{u}_0 = \hat{a}(\xi) + \hat{v}(\rho,\xi)$, where $ \hat{a}(\xi) = -  \int_0^\infty K_0(\rho'|\xi|)\hat{f}(\rho',\xi)\rho'd\rho',$
$\hat{v} \in \rho^2 L^2$, and both $\hat{a}$ and $\hat{v}$ decay rapidly in $\xi$. 

We have now constructed a right inverse for $N(\na \n)$ on all the different components of $1$-forms which lie in 
$\rho^{\delta - 2}L^2$, namely for $1$-forms with vanishing $n=0$ Fourier mode, for $1$-forms with only 
$d\rho$ and $\rho d\theta$ components and with only $n=0$ Fourier mode nonvanishing, and finally, by this last 
construction, for $1$-forms with $dy$ component lying in this $n=0$ Fourier mode subspace.  Inserting this
directly into the proof of \cite[Thm. 6.1]{Ma-edge}, and following that proof, we obtain a right parametrix $G_P$
for $P$, i.e.\ such that $P G_P = I - K_P$ on $L^2$.  More specifically, $G_P$ and $K_P$ are pseudodifferential edge 
operators of order $-2$ and $-\infty$, respectively; there is a precise description of the asymptotic behaviour of
the Schwartz kernels of these operators, and these determine the fine mapping properties of $P$. In fact, using
this asymptotic description exactly as in Proposition~\ref{pr:regedge}, we obtain that $P$ has closed range of
finite codimension in $L^2$; since we can identify its cokernel with the nullspace of $P$ on this same space,
which we know is trivial, we obtain that $P$ is surjective.  Using the precise vanishing order and description of the
leading coefficient in the expansion of the Schwartz kernel of $G_P$ as $\rho \to 0$, which is the same as the vanishing order 
and leading coefficient in the expansion for the right inverse of $N(\na \n)$, we have proved the first assertion in
the following 
\begin{Prop}
Let $f \in L^2$. Then the solution $\eta \in \calD_{\Fr}(P)$ to $Pu = f$ lies in $\rho^{1-\e}H^2_{\mathrm{e}}$ for all $\e > 0$
and it has no $\log \rho$ term in its asymptotic expansion as $\rho \to 0$. 

If $\eta \in \calD_\Fr(P)$ and $P \eta = f \in \calA_{\phg} \cap L^2$, then $\eta \in \calA_{\phg}$ near this edge and 
\begin{equation}
\eta \sim a(y) dy + c^{-+}(y)\, \eta^{-+}_{+1} + c^{+-}(y)\, \eta^{+-}_{-1} + \calO(\rho^\mu)
\label{eq:expedge}
\end{equation} for some $\mu > 1$, where $\eta^{-+}_{+1}$ and $\eta^{+-}_{-1}$ are as in \eqref{eq:solindcs} (with $\rho$ instead of $r$).
 Because of this, $\n \eta$, $\n d\eta$ and $d \delta \eta$
all lie in $L^2$. 
\label{pr:regedge2}
\end{Prop}
The remaining assertions are proved as follows. First, the fact that if $f$ is polyhomogeneous then the uniquely defined 
solution $\eta \in \calD_\Fr(P)$ is also polyhomogeneous is part of Proposition~\ref{pr:regedge} again, and represents a general
mapping property of pseudodifferential edge operators such as $G_P$. Once we know that $\eta$ is polyhomogeneous,
the initial terms in its expansion can be calculated directly from the equation $P \eta = f$ by inserting an arbitrary expansion
on the left for $\eta$ and matching terms with the given expansion for $f$. We have given all the calculations which
lead to the descriptions of the leading coefficients of $\eta$, and have explained why these lead to the final assertions
in the statement of this Proposition. This finishes the proof. 

Using these final assertions here, the proofs of Propositions \ref{pr:Laph0} and \ref{pr:Laph1} carry over verbatim and show 
the equivalence of the domains $\calD_{1,\Fr} = \calD_{1,\DN}$ for the operator $P$ on $M$, at least near the singular edges.

\subsection*{Computations near a singular vertex}
Now let $q$ be a singular vertex, so that as in \eqref{eq:metric2conemfd}, the metric has the form $g=dr^2 + \sn_{\K}^2 r \, h$, 
where the link $(N,h)$ is a spherical cone surface. Any $1$-form $\eta$ can be written in a neighbourhood of $q$ as 
$\eta = f\,dr + r \sigma$ where $\sigma$ has no $dr$ component. The indicial operator for $P$ is the same as the one for 
$\na \n$, and is given by
\begin{align*}
I(\na \n) \eta = &  r^{-2}\bigg[\left( \left(-(r\d_r)^2 - r\d_r + 2 + \Delta_{N,0} \right) f - 2 \delta_N \sigma \right) dr \\
& +  r \big( \left(-(r \d_r)^2 - r \d_r  + \Delta_{N,1}\right)\sigma - 2 d_N f \big)\bigg], 
\end{align*}
where $\Delta_{N,j}$ is the Laplacian on $j$-forms on $N$. 

Now decompose $f$ and $\sigma$ using the bases of eigenforms from Proposition \ref{prop:bashilb}.
Each summand in the coexact part of $\sigma$, i.e.\ each term $r^\mu \phi_j''$, yields the equation
\[
((r\d_r)^2 + r \d_r - \lambda_j) r^\mu = 0 \Longrightarrow \mu = -\frac12 \pm \frac12 \sqrt{1 + 4\lambda_j}.
\]
Similarly for $\lambda_0 = 0$, we obtain the equation $((r\d_r)^2 + r \d_r - 2) r^\mu = 0$, and the corresponding solutions 
are $\eta = rdr$ or $r^{-2}dr$.

When $\lambda_j \neq 0$, the terms $(a \psi_j + b \phi_j')r^\mu$ satisfy coupled equations which lead to the linear system
\[
\begin{pmatrix} \mu^2 + \mu - 2 - \lambda_j &  2 \sqrt{\lambda_j}  \\ 2 \sqrt{\lambda_j} & \mu^2 + \mu - \lambda_j
\end{pmatrix} \begin{pmatrix} a \\ b \end{pmatrix} = \begin{pmatrix} 0 \\ 0 \end{pmatrix},
\]
which we find, after a calculation, has solutions if and only if
\[
\mu = -\frac32 \pm \frac12 \sqrt{1 + 4 \lambda_j}, \qquad \frac12 \pm \frac12 \sqrt{1 + 4\lambda_j}.
\]

Except for the exceptional case $\lambda_0 = 0$, all other $\lambda_j$ are strictly greater than $1$ (in fact,
by Proposition \ref{pr:weiss} they are strictly greater than $2$, though we do not need to use this here). 
Using this bound, we see that the indicial roots in the critical interval $(-3/2,1/2]$ neatly separate into two groups, 
\[
-\frac32 + \frac12 \sqrt{1 + 4\lambda_j} \in (-\frac12, \frac12), \quad \mbox{and}
\quad
\frac12 - \frac12 \sqrt{1 + 4\lambda_j} \in (-\frac32, -\frac12),
\]
and these only occur when $\lambda_j \in (1, \frac{15}4)$. 

The Friedrichs extension allows the first group of solutions, but not the solutions which blow up faster than 
$r^{-1/2}$. Consider the case where $\lambda_j \neq 0$.  The corresponding solutions are multiples of
\[
\eta_j^+ = r^{-\frac32 + \frac12\sqrt{1+4\lambda_j}} \left( 2\sqrt{\lambda_j} \psi_j\, dr + (1 + \sqrt{1 + 4\lambda_j}) r \phi_j' \right).
\]
A computation then shows that $d\eta_j^+ = 0$. 

We have now assembled all the facts needed to prove Theorem \ref{thm:gauge}. Let $P$ denote the
Friedrichs extension $P_{\Fr}$. 

If $(M,g)$ is hyperbolic, then the usual integration by part is allowed for $\eta \in \calD_\Fr$, so the nullspace 
of $P = \na \n + 2$ is trivial. Hence in this case, $P: \calD_\Fr \to L^2$ is an isomorphism. Now suppose
that $f \in L^2 \cap \calA_{\phg}$ and let $\eta$ denote the unique solution in $\calD_\Fr$ to $P\eta = f$. 
Proposition \ref{pr:regedge2} shows that $\eta$ is polyhomogeneous near the singular edges, and its 
expansion \eqref{eq:expedge} implies that $\eta \in L^2_{\calD_N}$ near any singular vertex, where $\calD_N 
= \calD_{1,\Fr} = \calD_{1,DN}$.  Proposition \ref{pr:polysv} now implies that $\eta$ is also polyhomogeneous near each singular vertex. 

Since $\eta \in \calD_\Fr$, we automatically have that $\n \eta \in L^2$, and we have also seen in Proposition \ref{pr:regedge2} 
that $\n d\eta$ and $d\delta \eta$ lie in $L^2$ near the singular edges. It remains to check that this still holds near the 
singular vertices. But in any such neighbourhood, by the computations above it has an expansion 
$\eta = \sum c_j \eta_j^+ + \tilde{\eta}$ where $\tilde{\eta} \in r^2 s^2 H^2_\ie$; consequently 
$d\eta = d\tilde{\eta} \in rs H^1_\ie$, hence taking one more derivative shows that $\n d\eta \in L^2$, as required. 
Finally, $d \delta \eta = \Delta \eta - \delta d \eta \in L^2$ too.

If $(M,g)$ is Euclidean, then we can apply all the same reasoning so long as $P$ is still an isomorphism.  Because
of self-adjointness, it suffices to show that if $\eta$ is in the Friedrichs domain and $P\eta = 0$, then $\eta = 0$. 
Integrating by parts, we find first of all that $\nabla \eta = 0$.  If $(M,g)$ has no singular vertices, then since it admits
a global parallel vector field, it must have a cover isometric to $\R \times Y$ where $Y$ is a flat cone-surface.
Excluding this case, suppose that $(M,g)$ has at least one singular vertex. If $P\eta=0$, then $\eta$ is parallel but it is easy to see that there exists no non-trivial parallel $1$-form in the neighbourhood of a singular vertex, so $\eta=0$.

This completes the proof of Theorem \ref{thm:gauge} in all cases. 

\section{Infinitesimal rigidity}
We now prove our main result:
\setcounter{Thm}{0}
\begin{Thm}
Let $M$ be a closed, connected three-dimensional cone-manifold with all cone angles smaller than $2\pi$. If $M$ is hyperbolic, 
then $M$ is infinitesimally rigid relative to its cone angles, i.e.\ every angle-preserving infinitesimal deformation is trivial.
If $M$ is Euclidean, then every angle-preserving deformation also preserves the spherical links of the codimension $3$ 
singular points of $M$.

In particular, convex hyperbolic polyhedra are infinitesimally rigid relatively to their dihedral angles, while every 
dihedral angle preserving infinitesimal deformation of a convex Euclidean polyhedron also preserves the internal angles 
of the faces.
\end{Thm}

A sketch of the proof has already been indicated at the end of \S \ref{sec:gauge}. It is a variant of the Calabi-Weil method 
due to Koiso \cite{Koiso} in the smooth compact case, cf. also \cite{Besse} \S 12.H and \cite{MontcouqThese}.

Let $\dot{g}$ be an infinitesimal angle-preserving deformation of $M$, which we assume is in standard form in a 
neighbourhood of the singular locus. Proposition \ref{prop:stdl2} implies that $\dot{g}$ and $\n \dot{g}$ are in $L^2$ and 
that $\dot{g}$ is polyhomogeneous along $\Sigma$; in addition, $\dot{g}$ satisfies 
\begin{equation*}
L \dot{g} -2\delta^*B(\dot{g})=0 
\end{equation*}
where $L$ is as in  \eqref{eq:Lg2}. (The superscripts indicating the metric will now be omitted.)

Now define a new infinitesimal deformation $h = \dot{g} - 2 \delta^* \eta$, where $\eta$ is the solution of
$P\eta = B \dot{g}$ obtained from Theorem \ref{thm:gauge}.  By construction, $h$ is in Bianchi gauge, $B(h) = 0$,
and by the same Theorem, since $B\dot{g} \in L^2 \cap \calA_{\phg}$, we have as well that $h \in L^2 \cap \calA_{\phg}$ 
and satisfies
\begin{equation*}
Lh  = \na \n h + 2\K\,(h - (\tr h) g) = 0. 
\end{equation*}
In order to show that the original deformation $\dot{g}$ is trivial (in the hyperbolic case), we must show that $h$ vanishes.

As a first step, take the trace of this equation with respect to $g$ to get
\begin{equation*}
(\Delta_0 - 4\K) \tr h= 0. 
\end{equation*}
Since $\tr h \in \calD_\Fr(\Delta_0)$, we can integrate by parts to conclude that $\tr h = 0$ when $(M,g)$
is hyperbolic, and that $\tr h$ is constant when $(M,g)$ is Euclidean.  This last conclusion is not surprising: 
any constant multiple of the metric $\lambda g$ is an infinitesimal deformation of an Euclidean cone-manifold, 
corresponding to a dilation of the metric. Although these homothetic infinitesimal deformations are nontrivial
according to our definition,  they are not particularly interesting. Hence we replace $h$ by $h - \frac13 (\tr h) g$,
and assume henceforth that $\tr h = 0$ in the Euclidean case as well.

\subsection*{Bochner technique in the hyperbolic case}
The $2$-tensor $h$ satisfies the system 
\begin{equation}\label{eq:4}
\na \n h -2 h =0, \quad \delta h = 0, \qquad \tr h =0.
\end{equation}
We now change our point of view slightly and regard the symmetric $2$-tensor $h$ as a $T^*M$-valued $1$-form, 
i.e.\ a section of $\Lambda^1(M;T^*M)$. Coupling $d$ with the Levi-Civita connection on $T^*M$ yields a differential 
$d^\n$, which extends to act on $T^*M$-valued forms of all degrees; its adjoint is denoted $\delta^\n$. 
Note that if $\omega \in \Gamma(T^*M)$ is a $T^*M$-valued $0$-form, then $d^\n \omega = \n \omega \in 
\Gamma(T^*M\otimes T^*M)$, and reciprocally if $k \in \Gamma(T^*M\otimes T^*M)$ is a $T^*M$-valued $1$-form 
then $\delta^\n k = \na k$.  A classical Weitzenb\"{o}ck formula, see \cite{Besse}, states that
\begin{equation*}
\na \n h = (\delta^\n d^\n + d^\n \delta^\n) h + \K((\tr h) g -3h).
\end{equation*}
By the second equation in \eqref{eq:4}, $\delta^\n  h = \na h = \delta h = 0$ and also, $\tr h = 0$; using this and $\K = -1$ in the
first equation there, we get
\begin{equation}\label{eq:finalhypeq} 
\delta^\n d^\n h + h = 0.
\end{equation}

If we were to already know that $h \in \calD_{\Fr}(\na \n)$, then $d^\n h \in L^2$ and we could take the inner product 
of (\ref{eq:finalhypeq}) with $h$ and integrate by parts to conclude that $h \equiv 0$.  However, we must argue 
further to show that $d^\n h \in L^2$. The key is to use the decomposition of $\n \eta = d^\n \eta$ into symmetric 
and antisymmetric parts: $d^\n \eta = 
\delta^* \eta + \frac12 d\eta$, or equivalently, $2\delta^* \eta = 2 d^\n \eta - d \eta$. This shows that
\[
d^\n h = d^\n (\dot{g} - 2 \delta^* \eta) = d^\n (\dot{g} + d\eta) - 2  (d^\n)^2 \eta. 
\]
However, $\n \dot{g}$ and $\n d \eta$ both lie in $L^2$, hence so does $d^\n (\dot{g} + d\, \eta)$.  Furthermore, 
$(d^\n)^2$ is the curvature operator $\eta \mapsto - \K \, \eta \otimes \, \mbox{Id} = \eta \otimes \, \mbox{Id}$, 
which has order $0$, so trivially $(d^\n)^2 \eta \in L^2$. This proves that $d^\n h \in L^2$, as claimed. 

Now integrate by parts. The main term becomes
\[
\langle \delta^\n d^\n h, h\rangle = 
 \langle \delta^\n d^\n h ,  \dot{g} + d\eta \rangle - 2 \langle \delta^\n d^\n h, d^\n \eta \rangle.
\]
But we just showed that $\n (\dot{g} + d\eta) \in L^2$, and this is enough to deduce
\[
\langle \delta^\n d^\n h, \dot{g} + d\eta \rangle  = \langle d^\n h, d^\n (\dot{g} + d\eta) \rangle.
\]

It remains to integrate by parts in $\langle \delta^\n d^\n h, d^\n \eta \rangle$. The key once again is that $(d^\n)^2$ 
is a bounded operator of order zero. Since $\eta, \n \eta \in L^2$, we see that $\eta \in \calD_{\max}(\n)$, which
equals $\calD_{\min}(\n)$ by Proposition \ref{thm:stokes}.   Hence by definition, there exists a sequence $\eta_k \in \calC^\infty_0
(M \setminus \Sigma)$ which converges to $\eta$ in the $H^1$ topology. Thus
\[
\langle \delta^\n d^\n h, d^\n \eta \rangle = \lim_{k\to \infty} \langle \delta^\n d^\n h, d^\n \eta_k \rangle,
\]
and since $\eta_k$ is $\calC^\infty_0$ and $(d^\n)^2$ is bounded,
\begin{equation*}
\langle \delta^\n d^\n h, d^\n \eta_k \rangle = \langle d^\n h, (d^\n)^2 \eta_k \rangle  
\longrightarrow \langle  d^\n h, (d^\n)^2 \eta \rangle.
\end{equation*}
This gives $\langle \delta^\n d^\n h, d^\n \eta \rangle = \langle d^\n h, (d^\n)^2 \eta \rangle,$
so that finally,
\[
\langle (\delta^\n d^\n + 2) h, h \rangle = ||d^\n h||^2 + 2||h||^2 = 0.
\]
This implies that $h \equiv 0$, and hence $\dot{g} = 2\delta^* \eta$: the infinitesimal deformation is trivial. 

\subsection*{Bochner technique in the Euclidean case}
When $\K = 0$, the symmetric traceless $2$-tensor $h$ satisfies the system
\begin{equation*}
\na \n h  =0, \quad \delta h = 0, \qquad \tr h =0.
\end{equation*}
Precisely the same argument as above implies that
\begin{equation}\label{eq:closed}
d^\n h  = 0, \qquad \delta^\n h =0.
\end{equation}
This has a simple interpretation: the bundle $T^*M$ with its Levi-Civita connection is flat, so we may couple it to
the de Rham complex and define the twisted cohomology spaces $H^*(M,T^*M)$. Then (\ref{eq:closed}) shows
that $h$ is the harmonic representative of a cohomology class in $H^1(M,T^*M)$.

To proceed further, we must discuss some of the geometric aspects of the deformation theory of flat cone-manifolds. 
Let $M_\reg = M \setminus \Sigma$ denote the regular part of $M$. This is an incomplete smooth flat Riemannian manifold.  
Its Euclidean structure can be lifted to the universal cover $\widetilde{M_\reg}$. This gives rise to the {\it developing map} 
$\dev : \widetilde{M_\reg} \to \R^3$; this is a local isometry which is uniquely defined up to an overall Euclidean motion. 
Associated to $\dev$ is the {\it holonomy representation}  $\rho : \pi_1(M_\reg) \to \mathrm{Isom}(\R^3)$, which is defined 
up to conjugacy 
and satisfies $\dev(\gamma(x)) = \rho(\gamma) \cdot \dev(x)$ for $x\in \widetilde{M_\reg}$ and $\gamma \in 
\pi_1(M_\reg)$.  Since $M$ is a cone-manifold, if $\gamma$ is any loop going around a singular edge of $M$, then 
$\rho(\gamma)$ is a rotation by an angle equal to the dihedral angle of that edge (modulo $2\pi$); furthermore,
if $\gamma_1$ and $\gamma_2$ are loops going around adjacent singular edges, then the axes of the corresponding rotations 
$\rho(\gamma_1)$ and $\rho(\gamma_2)$ are concurrent. 

A first-order deformation of the Euclidean structure on $M$ induces a vector field $X$ on $\widetilde{M_\reg}$ which 
represents the infinitesimal motion of points under the developing map. This vector field is equivariant in the
sense that if $\pr: \widetilde{M_\reg} \to M_\reg$ is the covering map, then $\left. (\pr)_*\right|_x X  - 
\left. (\pr)_* \right|_{\gamma(x)} (X)$ is a local Killing field. The infinitesimal deformation of the Euclidean metric on 
$\widetilde{M_\reg}$ is the Lie derivative $\tilde{h} = \calL_X g_{\widetilde{M_\reg}}$, and this descends to the 
infinitesimal deformation $h$ on $M_\reg$. 

In addition, there is an infinitesimal deformation of the holonomy representation, $\dot \rho : \pi_1(M_\reg ) \to \mathfrak g$, 
where $\mathfrak g$ is the Lie algebra of $G=\mathrm{Isom}(\R^3)$. This satisfies a cocycle condition: 
$\dot \rho (\gamma_1 \gamma_2) = \dot \rho (\gamma_1) + Ad(\rho(\gamma_1))(\dot \rho (\gamma_2))$ for all 
$\gamma_1$, $\gamma_2$ in $\pi_1(M_\reg)$.  If the infinitesimal deformation is trivial, then $\dot \rho$ 
is the derivative of a one-parameter family $\rho_t$ of representations where the $\rho_t$ are all 
conjugate to one another in $G$ and $\rho_0 = \rho$.  In this case, $\dot \rho$ is a coboundary, i.e.\ there exists 
a $v \in \mathfrak g$ such that $\dot \rho (\gamma) = v - Ad(\rho(\gamma))(v)$ for all $\gamma$ in $\pi_1(M_\reg)$. 
In other words, the space of infinitesimal deformations modulo trivial deformations can be identified with the 
group cohomology  $H^1(\pi_1(M_{\reg}); Ad\circ \rho)$. 

Now, $\R^3 = G/K$, where $K$ is the maximal compact group consisting of all
rotations fixing a point. Hence there exists over $\R^3$ a canonical flat $\mathfrak g$-bundle, with $G$ acting on 
the left by the adjoint action. Pulling back by $\dev$ gives a flat $\mathfrak g$-bundle $\tilde{E}$ over $\widetilde{M_\reg}$,
which descends to a flat $\mathfrak g$-bundle $E$ over $M_\reg$. This is the bundle of 
(germs of) infinitesimal isometries of $M_\reg$.  There is an alternate definition of $E$ as the quotient of 
$\widetilde{M_\reg}\times \mathfrak{g}$ by the equivalence relation  $(x,v)\sim(\gamma(x),Ad(\rho(\gamma))(v))$. 
In any case, $E$ splits into two orthogonal subbundles, $E_1 \oplus E_2$, where $E_1 \simeq TM_\reg$ is the
bundle of infinitesimal translations, and $E_2$ has fiber at $p$ isomorphic to $\mathfrak{so}(3)$ and corresponds 
to infinitesimal rotations centered at $p$. The flat connection $D$ on $E$ does not preserve this splitting, but does
preserve the flat subbundle $E_1$. In fact, the restriction of $D$ to $E_1$ is just the Levi-Civita connection on $TM_\reg$. 
The same decomposition holds for the lifted bundle $\tilde{E}$ over $\widetilde{M_\reg}$.

We have introduced $E$ because there is an isomorphism 
\begin{equation}\label{eq:isom}
H^1(M_\reg ;E) \stackrel{\sim}{\longrightarrow} H^1(\pi_1(M_\reg);\mathrm{Ad}\circ\rho)
\end{equation}
between $E$-valued de Rham cohomology and group cohomology, given by integrating closed $E$-valued $1$-forms 
over loops in $M_\reg$.  If $\omega \in \Omega^1(M_\reg;E)$ is a closed $1$-form which corresponds to an infinitesimal 
deformation of $M_\reg$, then it can be lifted to a closed $1$-form $\tilde{\omega} \in \Omega^1(\widetilde{M_\reg};\tilde{E})$.
This lift is exact since $\widetilde{M_\reg}$ is simply connected,  so there is a section $\tilde{s}$ of $\tilde{E}$ such that 
$\tilde{\omega}=D\tilde{s}$. The real part of $\tilde{s}$, i.e.\ the projection of $\tilde{s}$ onto $\tilde{E_1}\simeq T\widetilde{M_\reg}$, 
is naturally identified with the equivariant vector field $X$ on $\widetilde{M_\reg}$ defined earlier. In particular,  the cohomology class 
of $\omega$ is determined by this vector field $X$.

Now we may proceed with the proof in the Euclidean case.  Using the canonical isomorphisms $T^*M_\reg \simeq TM_\reg 
\simeq E_1 \subset E$ and the fact that the flat connection on $E$ preserves $E_1$, we see that the infinitesimal deformation
$h$ is also a harmonic $E$-valued $1$-form.  As such, there exists a section of $\tilde{E}$, which can be taken to equal 
its real part $X_h$, such that the lift $\tilde{h}$ of $h$ to $\widetilde{M_\reg}$ is equal to $D\,X_h$.  Since $D$ is the 
Levi-Civita connection on $TM_\reg$ (or $T\widetilde{M_\reg}$), this can be written as $\tilde{h} = \n \alpha_h$ where 
$\alpha_h$ is the section of $T^*\widetilde{M_\reg}$ dual to $X_h$. As we did earlier, we can separate this into its symmetric and 
alternating parts and write $\tilde{h}= \delta^* \alpha_h + \frac 12 d \alpha_h$.  However, $h$ and $\tilde{h}$ are symmetric tensors,
so $d\alpha_h =0$ and $\tilde{h}= \delta^* \alpha_h =  \calL_{\frac 12 X_h} g_{\widetilde{M_\reg}}$. This shows that $\frac12 X_h$ 
is the equivariant vector field describing the infinitesimal motion of points under the developing map, and hence that
$\frac12 h$ represents the cohomology class in $H^1(\pi_1(M_{\reg};Ad\circ\rho) \simeq H^1(M_\reg;E)$ corresponding to the 
original first-order deformation of the holonomy representation. 

Using this we can now finish the proof.  Indeed, the first-order deformation of the holonomy representation is given by 
integration of $h$: $\dot \rho (\gamma) = \int_\gamma \frac12 h$ for all $\gamma \in \pi_1(M_\reg)$ (this is just
the isomorphism \eqref{eq:isom} between $H^1(\pi_1(M_\reg);\mathrm{Ad}\circ\rho)$ and $H^1(M_\reg;E)$). 
Since $h$ is $E_1$-valued, this shows that $\dot \rho : \pi_1(M_\reg) \to \mathfrak g$ has values in the subalgebra 
of infinitesimal translations. Therefore, to first order, the holonomy representation is modified by translations only.

If $q$ is a singular vertex in $M$, i.e.\ a codimension $3$ singular point, then its link $N_q$ is embedded isometrically 
(up to scaling) as the boundary of a sufficiently small metric ball around $q$. Denoting by $N_\reg$ the regular part of $N$, 
we obtain an embedding $i : N_\reg \to M_\reg$. The fundamental group of $N_\reg$ is generated by loops winding around the singular edges meeting at $q$, so there exists a point $p \in \R^3$, corresponding to $q$, such that for any $\gamma \in 
\pi_1(N_\reg)$, $p$ lies on the axis of $\rho\circ i_* (\gamma)$. This gives a canonical map 
\[
\rho\circ i_* : \pi_1(N_\reg) \to \mathrm{Fix}\,(p) \simeq \mathrm{Isom}\,(\Sph^2),
\]
which is exactly the holonomy representation of spherical structure on $N$.

However, we proved that the deformation of the holonomy representation $\dot \rho$ has values in the subalgebra of 
infinitesimal translations. This implies at once that the first-order deformation of the holonomy representation of 
$N_q$ is zero.  This proves that to first order, the links of the vertices of $M$ are preserved. This completes the proof.

\vskip 2cm

\bibliographystyle{abbrv}

\end{document}

%% file: glueing2.pdf_t
\begin{picture}(0,0)%
\includegraphics{./glueing2.pdf}%
\end{picture}%
\setlength{\unitlength}{3947sp}%
\begingroup\makeatletter\ifx\SetFigFont\undefined%
\gdef\SetFigFont#1#2#3#4#5{%
  \reset@font\fontsize{#1}{#2pt}%
  \fontfamily{#3}\fontseries{#4}\fontshape{#5}%
  \selectfont}%
\fi\endgroup%
\begin{picture}(12466,5896)(-193,-5429)
\put(3451,-4861){\makebox(0,0)[lb]{\smash{{\SetFigFont{20}{24.0}{\familydefault}{\mddefault}{\updefault}{\color[rgb]{0,0,0}$\Omega_2\subset X_q$}%
}}}}
\put(9976,-1561){\makebox(0,0)[lb]{\smash{{\SetFigFont{20}{24.0}{\familydefault}{\mddefault}{\updefault}{\color[rgb]{0,0,0}$q$}%
}}}}
\put(10726,-2011){\makebox(0,0)[lb]{\smash{{\SetFigFont{20}{24.0}{\familydefault}{\mddefault}{\updefault}{\color[rgb]{0,0,0}$q$}%
}}}}
\put(1501,164){\makebox(0,0)[lb]{\smash{{\SetFigFont{20}{24.0}{\familydefault}{\mddefault}{\updefault}{\color[rgb]{0,0,0}$\Omega_1\subset X_q$}%
}}}}
\put(7651,164){\makebox(0,0)[lb]{\smash{{\SetFigFont{20}{24.0}{\familydefault}{\mddefault}{\updefault}{\color[rgb]{0,0,0}$C_a(\Omega_1)\subset X$}%
}}}}
\put(8326,-2986){\makebox(0,0)[lb]{\smash{{\SetFigFont{20}{24.0}{\familydefault}{\mddefault}{\updefault}{\color[rgb]{0,0,0}$g\in G_q\subset G$}%
}}}}
\put(1801,-2986){\makebox(0,0)[lb]{\smash{{\SetFigFont{20}{24.0}{\familydefault}{\mddefault}{\updefault}{\color[rgb]{0,0,0}$g\in G_q$}%
}}}}
\put(9826,-4861){\makebox(0,0)[lb]{\smash{{\SetFigFont{20}{24.0}{\familydefault}{\mddefault}{\updefault}{\color[rgb]{0,0,0}$C_a(\Omega_2)\subset X$}%
}}}}
\end{picture}%

%% file: induction2.pdf_t
\begin{picture}(0,0)%
\includegraphics{./induction2.pdf}%
\end{picture}%
\setlength{\unitlength}{3947sp}%
\begingroup\makeatletter\ifx\SetFigFont\undefined%
\gdef\SetFigFont#1#2#3#4#5{%
  \reset@font\fontsize{#1}{#2pt}%
  \fontfamily{#3}\fontseries{#4}\fontshape{#5}%
  \selectfont}%
\fi\endgroup%
\begin{picture}(5402,2102)(300,-1562)
\put(4501,-61){\makebox(0,0)[lb]{\smash{{\SetFigFont{11}{13.2}{\familydefault}{\mddefault}{\updefault}{\color[rgb]{0,0,0}Length $2\pi$}%
}}}}
\put(4501,-1111){\makebox(0,0)[lb]{\smash{{\SetFigFont{11}{13.2}{\familydefault}{\mddefault}{\updefault}{\color[rgb]{0,0,0}Length $\alpha$}%
}}}}
\end{picture}%

%% file: twistb.pdf_t
\begin{picture}(0,0)%
\includegraphics{./twistb.pdf}%
\end{picture}%
\setlength{\unitlength}{3947sp}%
\begingroup\makeatletter\ifx\SetFigFont\undefined%
\gdef\SetFigFont#1#2#3#4#5{%
  \reset@font\fontsize{#1}{#2pt}%
  \fontfamily{#3}\fontseries{#4}\fontshape{#5}%
  \selectfont}%
\fi\endgroup%
\begin{picture}(8087,4221)(627,-4253)
\put(2244,-2652){\makebox(0,0)[lb]{\smash{{\SetFigFont{14}{16.8}{\familydefault}{\mddefault}{\updefault}{\color[rgb]{0,0,0}$p_1$}%
}}}}
\put(3985,-2678){\makebox(0,0)[lb]{\smash{{\SetFigFont{14}{16.8}{\familydefault}{\mddefault}{\updefault}{\color[rgb]{0,0,0}$l_1$}%
}}}}
\put(6975,-2526){\makebox(0,0)[lb]{\smash{{\SetFigFont{14}{16.8}{\familydefault}{\mddefault}{\updefault}{\color[rgb]{0,0,0}$p_2$}%
}}}}
\put(5065,-2659){\makebox(0,0)[lb]{\smash{{\SetFigFont{14}{16.8}{\familydefault}{\mddefault}{\updefault}{\color[rgb]{0,0,0}$l_2$}%
}}}}
\end{picture}%

%% file: split2.pdf_t
\begin{picture}(0,0)%
\includegraphics{./split2.pdf}%
\end{picture}%
\setlength{\unitlength}{3947sp}%
\begingroup\makeatletter\ifx\SetFigFont\undefined%
\gdef\SetFigFont#1#2#3#4#5{%
  \reset@font\fontsize{#1}{#2pt}%
  \fontfamily{#3}\fontseries{#4}\fontshape{#5}%
  \selectfont}%
\fi\endgroup%
\begin{picture}(3902,3322)(2175,-4037)
\put(3001,-886){\makebox(0,0)[lb]{\smash{{\SetFigFont{11}{13.2}{\familydefault}{\mddefault}{\updefault}{\color[rgb]{0,0,0}$\alpha$}%
}}}}
\put(5326,-886){\makebox(0,0)[lb]{\smash{{\SetFigFont{11}{13.2}{\familydefault}{\mddefault}{\updefault}{\color[rgb]{0,0,0}$\pi+\alpha/2$}%
}}}}
\put(4501,-886){\makebox(0,0)[lb]{\smash{{\SetFigFont{11}{13.2}{\familydefault}{\mddefault}{\updefault}{\color[rgb]{0,0,0}$\pi+\alpha/2$}%
}}}}
\end{picture}%